\documentclass[reqno]{amsart}
\usepackage{CJK}
\usepackage{hyperref}
\usepackage{cite}
\usepackage{amsmath}
\usepackage{mathrsfs}
\usepackage{xcolor}
\usepackage{bbm}
\usepackage{pifont}
\usepackage{bbding}
\usepackage{amsmath, esint}
\usepackage{mathrsfs}
\usepackage{amsfonts}
\usepackage{amssymb}
\usepackage{amsfonts,amssymb,amsmath,indentfirst,cases,amsthm}
\usepackage{epsfig}
\usepackage[numbers,sort&compress]{natbib}

\makeatletter
\@namedef{subjclassname@2020}{%
	\textup{2020} Mathematics Subject Classification}
\makeatother
\numberwithin{equation}{section}

\newtheorem{theorem}{Theorem}[section]
\newtheorem{lemma}[theorem]{Lemma}
\newtheorem{definition}[theorem]{Definition}

\newtheorem{remark}[theorem]{Remark}
\newtheorem{corollary}[theorem]{Corollary}

\allowdisplaybreaks

\newcommand{ \mint }{ {\int\hspace{-0.38cm}-}}

\begin{document}
	
\title[\hfil Weak Harnack inequalities for nonlocal double phase problems]{On the weak Harnack inequalities for nonlocal double phase problems}

\author[Y. Fang and C. Zhang  \hfil \hfilneg]{Yuzhou Fang and Chao Zhang$^*$}

\thanks{$^*$Corresponding author.}

\address{Yuzhou Fang \hfill\break School of Mathematics, Harbin Institute of Technology, Harbin 150001, China}
\email{18b912036@hit.edu.cn}

\address{Chao Zhang  \hfill\break School of Mathematics and Institute for Advanced Study in Mathematics, Harbin Institute of Technology, Harbin 150001, China}
\email{czhangmath@hit.edu.cn}


\maketitle

\begin{abstract}
This paper is devoted to studying the weak Harnack inequalities for nonlocal double phase functionals by using expansion of positivity, whose prototype is
$$
\iint_{\mathbb{R}^n\times\mathbb{R}^n} \left(\frac{|u(x)-u(y)|^p}{|x-y|^{n+sp}}+a(x,y)\frac{|u(x)-u(y)|^q}{|x-y|^{n+tq}}\right) \,dxdy
$$
with $a\ge0$ and $0<s\le t<1<p\le q$. The core of our approach is to establish several measure theoretical estimates based on the nonlocal Caccioppoli-type inequality, where the challenges consist in controlling subtle interaction between the pointwise behaviour of modulating coefficient and the growth exponents. Meanwhile, a quantitative boundedness result on the minimizer of such functionals is also discussed.
\end{abstract}

\section{Introduction}
\label{sec1}

In this paper, we are interested in the weak Harnack inequalities for the minimizer of nonlocal double phase functionals, switching drastically between two diverse phases according to whether the coefficient $a$ is zero or not, of the following form
\begin{align}
\label{main}
\mathcal{F}(u;\Omega):=\iint_{\mathcal{C}_{\Omega}} \big[|u(x)-u(y)|^pK_{sp}(x,y)+a(x,y)|u(x)-u(y)|^qK_{tq}(x,y)\big]\,dxdy,
\end{align}
where $\mathcal{C}_{\Omega}=(\mathbb{R}^n\times\mathbb{R}^n)\setminus((\mathbb{R}^n\setminus\Omega)\times(\mathbb{R}^n\setminus\Omega))$,
\begin{equation}
\label{sp}
0<s\le t<1,   \quad    1<p\le q<\infty,
\end{equation}
and $\Omega$ is an open bounded subset of $\mathbb{R}^n$ ($n\ge2$). Throughout this manuscript, the kernel functions $K_{sp},K_{tq}:\mathbb{R}^n\times\mathbb{R}^n\rightarrow[0,\infty)$ are supposed to satisfy
\begin{equation}
\label{k}
\frac{\Lambda^{-1}}{|x-y|^{sp}}\le K_{sp}(x,y)\le\frac{\Lambda}{|x-y|^{sp}}, \quad  \frac{\Lambda^{-1}}{|x-y|^{tq}}\le K_{tq}(x,y)\le\frac{\Lambda}{|x-y|^{tq}}
\end{equation}
with a number $\Lambda\ge1$, and the basic assumptions on the modulating coefficient $a:\mathbb{R}^n\times\mathbb{R}^n\rightarrow\mathbb{R}$ are measurable, symmetric and nonnegative:
\begin{equation}
\label{an}
0\le a(x,y)=a(y,x)  \quad \text{in } \mathbb{R}^n\times\mathbb{R}^n.
\end{equation}
Additionally, when we investigate weak Harnack estimates, we also require
\begin{equation}
\label{a}
|a(x_1,y_1)-a(x_2,y_2)|\le [a]_\alpha(|x_1-x_2|+|y_1-y_2|)^\alpha  % \quad\text{in }
\end{equation}
with $\alpha\in(0,1]$ in $\mathbb{R}^n\times\mathbb{R}^n$.

The regularity theory on nonlocal problems has become an active realm of research in the past two decades. For the standard fractional Laplace equation, Caffarelli and Silvestre \cite{CS07} showed Harnack inequality by the extension arguments; see \cite{CCV11} for H\"{o}lder continuity of solutions to the parabolic version. Afterward, Di Castro, Kuusi and Palatucci \cite{DKP14, DKP16} extended the De Giorgi-Nash-Moser technology to the nonlinear fractional $p$-Laplacian case, and obtained the local behaviors including H\"{o}lder, Harnack and weak Harnack estimates, where the key tool used is a logarithmic type inequality. One can refer to \cite{KW, KMS15, KLL23, KKL19, KKP17, Coz, Liao, Pra24} for more topics, such as measure data problems, nonlocal Wiener criterion, equivalence of various solutions, non-homogeneous functionals with lower order terms, parabolic equations and so on.

In recent years, some regularity results mentioned above have been developed for the nonlocal functionals or equations exhibiting non-standard growth. For instance, Byun, Kim and Ok \cite{BKO} proved the local boundedness and H\"{o}lder continuity for the nonlocal $G$-Laplace equations with Orlicz growth by exploiting the methods introduced by \cite{DKP16}; see also \cite{CKW22} for different approaches. On the other hand, for such an equation, Fang and Zhang \cite{FZ23} employed the De Giorgi iteration theory to establish the Harnack type inequality based on the weak Harnack estimate and boundedness of subsolutions. Subsequently, this result was improved by \cite{CKW23, BKS23}, in which they also analysed the robustness of the full Harnack estimate as $s\rightarrow1$. In addition, regarding the mixed local and nonlocal functionals with nonuniform growth, De Filippis and Mingione \cite{DeFM24} studied the $C^{1, \alpha}_{\rm loc}$ property through the perturbation argument, because the local $p$-Laplace operator refines the regularity directly.

For what concerns our main object in this work, the nonlocal double phase problem as a typical model with nonstandard growth has attracted increasing attention. De Filippis and Palatucci \cite{DP19} first considered this kind of equation and explored the viscosity theory, and then a nonlocal self-improving attribute for bounded weak solutions was deduced in \cite{SM22}. From the perspective of the concept of different solutions, Fang and Zhang \cite{FZ23} discussed the inner relationship between viscosity and weak solutions, and showed the H\"{o}lder continuity for weak solutions. Here we remark that the case under the central assumption $s\ge t$, which means that the $q$-growth term is a lower order term, was completed by the three papers above. On the contrary, Byun, Ok and Song \cite{BOS22} dealt with the complementary case $s\le t$, and, under the situation $0\le a\in L^\infty(\mathbb{R}^n\times\mathbb{R}^n)$, for minimizers/weak solutions gained the local H\"{o}lder regularity by considering carefully the interplay among growth powers, differentiability orders and coefficient $a$ with $tq\le sp+\alpha$. Such results on the nonlocal functional \eqref{main} can be compared sharply with the local analogues. When it comes to the classical double phase functional
\begin{equation}
\label{F}
\int_\Omega|Du|^p+a(x)|Du|^q\,dx,
\end{equation}
it is verified in \cite{BCM15,CM15a} that the bounded minimizers are of class $C^{0, \beta}_{\rm loc}(\Omega)$, provided the gap between growth exponents $p$ and $q$ is small enough, i.e., $q\le p+\alpha$ with $\alpha$ being the H\"{o}lder continuity index of $a$. In particular, De Filippis and Mingione \cite{DeFM23} proved the local gradient H\"{o}lder regularity for minimizers of nonuniformly elliptic integrals, that are not necessarily equipped with a Euler-Lagrange equation, which solved a longstanding and classical open problem in the regularity theory of variational integrals and elliptic equations. More investigations into local and nonlocal problems with nonstandard growth can be found in for example \cite{CM15b, BLSa, Gia, DeFM24, BM20, Mar89, HO22, BHK21} and references therein.

The main goal of this article is to establish the weak Harnack inequalities for \eqref{main}. To the best of our knowledge, the issues on Harnack type estimate or even on weak Harnack type inequality %have remained unsolved for the nonlocal double phase functionals/equations up to now.
remain unsolved for the nonlocal double phase functionals/equations, although extensive research on such problems has been carried out up to now. To justify our results, we make use of the expansion of positivity method to derive some measure theoretical information just from the Caccioppoli estimate (Lemma \ref{lem3-2}). Then weak inequalities follow from basic computations and without an application of Moser iteration that was however utilized to get such inequalities for the fractional $p$-Laplace equations in \cite[Theorem 1.2]{DKP14}. Furthermore, these processes circumvent the complicated log-estimate paralleling \cite[Lemma 5.1]{BOS22}, and the functional \eqref{main} does not come into play at all except showing the energy inequality. Finally, let us comment that the approach developed for the nonlocal $p$-Laplacian scenario could not plainly work here because of the nonuniform ellipticity of \eqref{main}. To overcome this difficulty, when inferring various measure estimates, we have to control the delicate interaction between the possibly degenerate part of the energy $a(x,y)|u(x)-u(y)|^qK_{tq}$ and the non-degenerate one $|u(x)-u(y)|^pK_{sp}$, via applying a precise comparison scheme to distinguish two phases.

Before stating our main contributions, we introduce some terminologies as follows. For $k\in\mathbb{R}$, denote the functions
\begin{equation}
\label{hr}
h_r(k):=\frac{k^{p-1}}{r^{sp}}+a^-_r\frac{k^{q-1}}{r^{tq}}
\end{equation}
and
\begin{equation}
\label{Hr}
H_r(k)=\frac{k^p}{r^{sp}}+a^-_r\frac{k^q}{r^{tq}}
\end{equation}
with
\begin{equation}
\label{ar}
a^-_r:=\inf_{B_r(x_0)\times B_r(x_0)} a(x,y).
\end{equation}
We use the symbols $H^{-1}_r$ and $h^{-1}_r$ to represent the inverses of $H_r$ and $h_r$, respectively. Owing to the occurrence of variable coefficient $a$, we have to consider the nonlocal tail with weight given as
$$
\mathrm{Tail}_a(u_-;x_0,r):=\sup_{y\in B_{r}(x_0)}\int_{\mathbb{R}^n\setminus B_{r}(x_0)}a(x,y)\frac{u_-^{q-1}(x)}{|x-x_0|^{n+tq}}\,dx.
$$
If the function $a(\cdot,\cdot)\equiv1$ such tail becomes the usual tail, defined by
$$
\mathrm{Tail}(u_-;x_0,r):=\int_{\mathbb{R}^n\setminus B_{r}(x_0)}\frac{u_-^{p-1}(x)}{|x-x_0|^{n+sp}}\,dx.
$$

%\begin{equation*}
%H_r(t)=\frac{t^p}{r^{sp}}+a^-_r\frac{t^q}{r^{tq}}.
%\end{equation*}
We are now in a position to state the first result on the weak Harnack inequalities.

\begin{theorem}
\label{Thm1}
Let the conditions that \eqref{sp}--\eqref{a}, and
\begin{equation}
\label{thm1}
\alpha<tq\le sp+\alpha
\end{equation}
be satisfied. Let also $u\in\mathcal{A}(\Omega)\cap L^{p-1}_{sp}(\mathbb{R}^n)\cap L^{q-1}_{a,tq}(\Omega,\mathbb{R}^n)$ that is nonnegative in a ball $B_{4R}(x_0)\subset \Omega$ with $R\le1$ be a locally bounded minimizer of \eqref{main}. Then there are two constants $\epsilon\in(0,1)$ and $C\ge1$, both depending on $\textbf{data}(B_{4R})$, such that
\begin{equation}
\label{thm2}
\left(\mint_{B_R(x_0)}u^\epsilon\,dx\right)^\frac{1}{\epsilon}\le C\inf_{B_R(x_0)}u+Ch_{4R}^{-1}\big(\mathrm{Tail}(u_-;x_0,4R)+\mathrm{Tail}_a(u_-;x_0,4R)\big).
\end{equation}
\end{theorem}

\begin{remark}
Correspondingly, for the functional \eqref{F}, Baroni, Colombo and Mingione \cite{BCM15} concluded
$$
\left(\mint_{B_R(x_0)}u^{t_-}\,dx\right)^\frac{1}{t_-}\le C\inf_{B_R(x_0)}u, \quad \textmd{for some }  t_->0,   %\mint_{B_{2R}(x_0)}u^{t_-}
$$
under the hypotheses that $0\le a(\cdot)\in C^{0,\alpha}(\Omega)$ and $q\le p+\alpha$. If the minimizer of \eqref{main} is globally nonnegative, then \eqref{thm2} is reduced to the classical version. It is worth pointing out that the extra condition $\alpha<tq$ in \eqref{thm1} is employed to guarantee the improper integral $\int_{\mathbb{R}^n\setminus B_{2R}(x_0)}|x-x_0|^{-(n+tq-\alpha)}\,dx$ converges. In contrast to the local double phase situation, we know the fact that $\alpha<q$ is always true ($1<q$), which can be regarded as the limiting case that $t$ goes to 1.
\end{remark}

The upcoming theorem refines the weak Harnack inequality above from two aspects. One is that the integral exponent can be explicitly exhibited; the other is that the integral power has been improved to $p-1$ or $q-1$.

\begin{theorem}
\label{Thm2}
Assume that $u\in\mathcal{A}(\Omega)\cap L^{p-1}_{sp}(\mathbb{R}^n)\cap L^{q-1}_{a,tq}(\Omega,\mathbb{R}^n)$, nonnegative in a ball $B_{4R}(x_0)\subset \Omega$ with $R\le1$, is a locally bounded minimizer of \eqref{main}. Then one can find a constant $\gamma\in(0,1)$ depending upon $\textbf{data}(B_{4R})$ such that
\begin{align*}
&\quad\inf_{B_R(x_0)}u+h_{4R}^{-1}\big(\mathrm{Tail}(u_-;x_0,4R)+\mathrm{Tail}_a(u_-;x_0,4R)\big)\\
&\ge\gamma\min\left\{(u^{p-1})_{B_{2R}}^\frac{1}{p-1},
(u^{q-1})_{B_{2R}}^\frac{1}{q-1}\right\}
\end{align*}
with the integral average
$
(v)_{B_{2R}}:=\mint_{B_{2R}(x_0)}v\,dx,
$
provided the hypotheses \eqref{sp}--\eqref{a} and \eqref{thm1} hold true.
\end{theorem}

The third weak Harnack estimate discloses the positivity contribution stemming from the long-range attribute of the minimizers to \eqref{main}, in addition to local positivity, which needs one precondition more constricted than preceding theorems that the coefficient $a(\cdot, \cdot)$ has a positive lower bound. Nevertheless, this coefficient still \textbf{may be unbounded from above} now.

\begin{theorem}
\label{Thm3}
Let the assumptions that \eqref{sp}--\eqref{a}, \eqref{thm1} and $a(\cdot,\cdot)\ge\lambda>0$ be in force. Suppose that $u\in\mathcal{A}(\Omega)\cap L^{p-1}_{sp}(\mathbb{R}^n)\cap L^{q-1}_{a,tq}(\Omega,\mathbb{R}^n)$ is a locally bounded minimizer of \eqref{main}, which is nonnegative in a ball $B_{4R}(x_0)\subset \Omega$ with $R\le1$. Then there exists a constant $\gamma\in(0,1)$, depending on $\textbf{data}(B_{4R})$ and $\lambda$, such that
\begin{align*}
&\quad\inf_{B_R(x_0)}u+h_{4R}^{-1}\big(\mathrm{Tail}(u_-;x_0,4R)+\mathrm{Tail}_a(u_-;x_0,4R)\big)\\
&\ge\gamma\min\left\{\mathrm{Tail}_p(u_+;x_0,2R),
\left(a^-_{4R}\right)^\frac{-1}{q-1}\mathrm{Tail}_q(u_+;x_0,2R)\right\},
\end{align*}
where
$$
\mathrm{Tail}_p(u_+;x_0,r)=\left(r^{sp}\int_{\mathbb{R}^n\setminus B_{r}(x_0)}\frac{u_+^{p-1}(x)}{|x-x_0|^{n+sp}}\,dx\right)^\frac{1}{p-1}
$$
and
$$
\mathrm{Tail}_q(u_+;x_0,r)=\left(r^{tq}\int_{\mathbb{R}^n\setminus B_{r}(x_0)}\frac{u_+^{q-1}(x)}{|x-x_0|^{n+tq}}\,dx\right)^\frac{1}{q-1}.
$$
\end{theorem}

Let us mention that this kind of weak Harnack inequality above, displaying strong nonlocality, is first discovered by Liao \cite{Liao24} for the nonlocal $p$-Laplace equation (the $q$-term does not occur at all in this case), which is a surprising result to some extent. Besides, in the previous three theorems, we assume {\it a priori} the minimizer of \eqref{main} is locally bounded in $\Omega$. The local boundedness of minimizers could be verified even without H\"{o}lder regularity of the coefficient $a$, which is stated as follows.

\begin{theorem}[Boundedness]
\label{Thm4}
Suppose that \eqref{sp}--\eqref{an} as well as $0\le a(x,y)\in L^\infty_{\rm loc}(\Omega\times\Omega)$ are in force. If
\begin{align}
\label{thm4}
\begin{cases}
\frac{q}{p} \leq 1+\frac{sp}{N-sp},  \quad &\text{for  } sp< n,\\[2mm]
p\le q<\infty,        \quad&\text{for  } sp\ge n,
\end{cases}
\end{align}
then each minimizer $u \in \mathcal{A}(\Omega)\cap L_{sp}^{p-1}\left(\mathbb{R}^n\right) \cap L_{a,tq}^{q-1}\left(\Omega,\mathbb{R}^n\right)$ of the functional \eqref{main} is locally bounded in $\Omega$.
\end{theorem}

Here we only demand the coefficient $a$ is locally bounded instead of globally bounded in \cite[Theorem 1.1]{BOS22}. After we show the minimizers of \eqref{main} have local boundedness, we can conduct further an exact supremum estimate stated below.

\begin{theorem} [Sup-estimate]
\label{Thm5}
Let $B_r(x_0)\subset\Omega$ and $u\in\mathcal{A}(\Omega)\cap L^{p-1}_{sp}(\mathbb{R}^n)\cap L^{q-1}_{a,tq}(\Omega,\mathbb{R}^n)$ be a locally bounded minimizer of \eqref{main}. Then under the hypotheses of \eqref{a} and $tq\le sp+\alpha$, we infer that
$$
\sup_{B_{\frac{r}{2}}(x_0)}u\le C_\delta H^{-1}_r\left(\mint_{B_r(x_0)}H_r(u_+)\,dx\right)+\delta h^{-1}_r\left(T\left(u_+;x_0,\frac{r}{2}\right)\right)
$$
for any $\delta\in(0,1]$, where the positive constant $C_\delta$ depends on $\textbf{data}(B_r)$ and $\delta$, and the term $T\Big(u_+;x_0,\frac{r}{2}\Big)$ is denoted by
\begin{equation}
\label{6-1-2}
T\Big(u_+;x_0,\frac{r}{2}\Big):=\int_{\mathbb{R}^n\setminus B_{\frac{r}{2}}(x_0)}\frac{u^{p-1}_+(x)}{|x-x_0|^{n+sp}}\,dx+\sup_{y\in B_r(x_0)}
\int_{\mathbb{R}^n\setminus B_{\frac{r}{2}}(x_0)}a(x,y)\frac{u^{q-1}_+(x)}{|x-x_0|^{n+tq}}\,dx.
\end{equation}
\end{theorem}

%\begin{remark}
%For the case $s>t$, local boundedness can be obtained under \eqref{eqy19}, \eqref{eqy112} by checking the proof of Theorem \ref{Th11}. Meanwhile, following the proof of Theorem \ref{Th12} and making a few slight modifications, we can deduce, under the same preconditions of Theorem \ref{Th12}, that weak solutions are also of the class $C^{0,\beta}_{\rm loc}(\Omega')$ with some $\beta\in\left(0,\frac{\min\{sp,tq\}}{q-1}\right)$.
%depending on $\mathrm{\mathbf{data}}$.  (except $(1.4)_2$)
%\end{remark}
Finally, according to the statement in \cite[Section 3]{BOS22}, we know that all the theorems above are also valid for the weak solutions to the following nonlocal double phase equation
\begin{align*}
&\mathrm{P.V.}\int_{\mathbb{R}^n}|u(x)-u(y)|^{p-2}(u(x)-u(y))K_{sp}(x,y)\,dy \\
&\quad+\mathrm{P.V.}\int_{\mathbb{R}^n}a(x,y)
|u(x)-u(y)|^{q-2}(u(x)-u(y))K_{tq}(x,y)\,dy=0  \quad \text{in } \Omega,
\end{align*}
where $\mathrm{P.V.}$ stands for the Cauchy principal value.

This paper is organized as follows. In Section \ref{sec2} we introduce some basic notations, notions and useful technical lemmas and prove the nonlocal Caccioppoli-type estimates in Section \ref{sec3}. Section \ref{sec4} is dedicated to exploring some results on the expansion of positivity, and subsequently we provide the proof of weak Harnack inequalities provided in Theorems \ref{Thm1}--\ref{Thm3} in Section \ref{sec5}. At last, the local boundedness along with the sup-estimate of minimizers to \eqref{main} are established in Section \ref{sec6}.

\section{Preliminaries}
\label{sec2}

We in this part shall give some concepts and notations as well as auxiliary results. Throughout this manuscript, denote by $C$ a positive constant that may vary from line to line. Relevant dependencies on parameters are emphasized by using parentheses, i.e., $C(p,q,n)$ means $C$ depends on $p,q,n$. Let $B_r(x_0):=\{x\in\mathbb{R}^n: |x-x_0|<r\}$ represent a ball with radius $r>0$ and center $x_0$. If not important or clear from the context, we will omit the center of the ball as $B_r=B_r(x_0)$. To shorten the notations, for any subset $D\subset \Omega$ we let
\begin{equation*}
\mathbf{data}(D):=\begin{cases}
n,p,q,s,t,\Lambda,\alpha,[a]_\alpha,\|u\|_{L^\infty(D)},   \quad &\text{if } sp\le n,\\[2mm]
n,p,q,s,t,\Lambda,\alpha,[a]_\alpha,[u]_{W^{s,p}(D)},   \quad &\text{if } sp> n.
\end{cases}
\end{equation*}

For $s\in(0,1)$ and $p\ge1$, the fractional Sobolev space $W^{s,p}(\Omega)$ is defined by
$$
W^{s,p}(\Omega):=\left\{u\in L^p(\Omega)\Bigg|[u]_{W^{s,p}(\Omega)}:=\left(\int_\Omega\int_\Omega\frac{|u(x)-u(y)|^p}{|x-y|^{n+sp}}\,dxdy\right)^\frac{1}{p}<\infty\right\},
$$
equipped with the norm $\|u\|_{W^{s,p}(\Omega)}=\|u\|_{L^p(\Omega)}+[u]_{W^{s,p}(\Omega)}$. Let
$$
\varepsilon(u;\Omega):=\iint_{\mathcal{C}_\Omega}\frac{|u(x)-u(y)|^p}{|x-y|^{n+sp}}+a(x,y)\frac{|u(x)-u(y)|^q}{|x-y|^{n+tq}}\,dxdy,
$$
and define a function space related to the minimizers of \eqref{main} as below
$$
\mathcal{A}(\Omega):=\left\{u:\mathbb{R}^n\rightarrow\mathbb{R} \,\Big|\, u|_\Omega\in L^p(\Omega) \ \text{and } \varepsilon(u;\Omega)<\infty\right\}.
$$
Here it is easy to see via the previous definitions that $\mathcal{A}(\Omega)\subset W^{s,p}(\Omega)$. Moreover, due to fractional Sobolev embedding it holds that $\mathcal{A}(\Omega)\subset L^q(\Omega)$, if $p\le q\le \frac{np}{n-sp}$ for $sp<n$ or $p\le q<\infty$ for $sp\ge n$. Now we state the definition of minimizer to the functional \eqref{main}.

\begin{definition}
\label{def1}
We call a function $u\in\mathcal{A}(\Omega)$ a minimizer of \eqref{main} whenever
$$
 \mathcal{F}(u;\Omega)\le\mathcal{F}(v;\Omega)
$$
for each $v\in\mathcal{A}(\Omega)$ with $v=u$ a.e. in $\mathbb{R}^n\setminus\Omega$.
\end{definition}

It is noteworthy that the existence of minimizer of \eqref{main} has been justified in \cite[Section 3]{BOS22}. Taking the nonlocal characteristic of the functional \eqref{main} into account, we next introduce the tail space given as
$$
L^{p-1}_{sp}(\mathbb{R}^n):=\left\{u:\mathbb{R}^n\rightarrow\mathbb{R}\,\Bigg|\,\int_{\mathbb{R}^n}\frac{|u(x)|^{p-1}}{(1+|x|)^{n+sp}}\,dx<\infty\right\}.
$$
Furthermore, owing to the $q$-growth term of \eqref{main} perturbed by the modulating coefficient $a$, we have to consider a ``tail space with weight" as follows,
$$
L^{q-1}_{a,tq}(\Omega,\mathbb{R}^n):=\left\{u: \mathbb{R}^n\rightarrow\mathbb{R}\,\Bigg|\,\sup_{x\in \Omega}\int_{\mathbb{R}^n}a(x,y)\frac{|u(y)|^{p-1}}{(1+|y|)^{n+sp}}\,dy<\infty\right\}.
$$%\underset{x\in \Omega}{{\rm ess}\sup}
We can readily find the various tails defined in the Introduction are finite according to the definitions of tail spaces.

In the sequel, let us collect three auxiliary conclusions involving inequality concerning fractional Sobolev functions and iteration lemmas.

\begin{lemma}[\cite{BOS22}]
\label{lem2-2}
Let $p, q, s, t$ satisfy \eqref{sp} and \eqref{thm4}. Then there exists a constant $C=C(n, p, q, s,t)>0$ such that, for each $f \in W^{s, p}(B_r)$,
\begin{align*}
\mint_{B_r}\left|\frac{f}{r^s}\right|^p+L_0 \left|\frac{f}{r^ t}\right|^q\,dx
& \leq C L_0 r^{(s-t) q}
\left(\mint_{B_r}\int_{B_r}\frac{|f(x)-f(y)|^p}{|x-y|^{n+sp}}\,dxdy\right)^{\frac{q}{p}}\\
&\quad+C \left(\frac{|\operatorname{supp} f|}{\left|B_r\right|}\right)^\frac{sp}{n}\mint_{B_r}\int_{B_r}\frac{|f(x)-f(y)|^p}{|x-y|^{n+sp}}\,dxdy \\
&\quad +C\left(\frac{|\operatorname{supp} f|}{|B_r|}\right)^{p-1} \mint_{B_r}\left|\frac{f}{r^s}\right|^p+L_0\left|\frac{f}{r^t}\right|^q\,dx
\end{align*}
with any $L_0>0$.
\end{lemma}

\begin{lemma}[\cite{DeFM24}]
\label{iteration}
Let $h:[\rho_1,\rho_2]\rightarrow\mathbb{R}$ be a nonnegative and bounded function. Suppose
$$
h(t)\le \theta h(s)+\sum^k_{i=1}\frac{a_i}{(s-t)^{p_i}}+b
$$
is valid with $\theta\in(0,1)$ and $a_i,p_i,b\ge0$ ($i\in \mathbb{N}$) if $\rho_1\le t<s\le \rho_2$. Then there holds that
$$
h(\rho_1)\le C\sum^k_{i=1}\frac{a_i}{(\rho_2-\rho_1)^{p_i}}+Cb
$$
for some constant $C=C(\theta,p_i)>0$.
\end{lemma}

\begin{lemma}[\cite{G03}]
\label{lem-2-3}
Let $\left\{Y_{j}\right\}_{j\in \Bbb{N}}$ be a sequence of nonnegative
numbers fulfilling% the recursive inequality
\begin{align*}
Y_{j+1} \leq K b^{j} Y_{j}^{1+\delta}, \quad  j=0,1,2,
\ldots,
\end{align*}
where $\delta,K>0, b>1$ are given numbers. If
$$
Y_{0} \leq  K^{-\frac{1}{\delta}} b^{-\frac{1}{\delta^2}}
$$
then $Y_{j} \leq 1$ for large $j\in \mathbb{N}$. Moreover,
$Y_{j}$ converges to 0 as $j\rightarrow\infty$.
\end{lemma}

\section{Energy estimates}
\label{sec3}

This section aims at establishing the Caccioppoli-type inequalities of the nonlocal version that contains all the information needed to conclude the weak Harnack estimates as well as boundedness. The first one is as follows.

\begin{lemma}
\label{lem3-1}
Let $u\in\mathcal{A}(\Omega)\cap L^{p-1}_{sp}(\mathbb{R}^n)\cap L^{q-1}_{a,tq}(\Omega,\mathbb{R}^n)$ be a minimizer of the functional \eqref{main}. Then for any ball $B_R:=B_R(x_0)\subset\Omega$ and $0<\rho<r\le R$, we infer that
\begin{align*}
&\quad\int_{B_\rho}\int_{B_\rho}|w_\pm(x)-w_\pm(y)|^pK_{sp}(x,y)+a(x,y)|w_\pm(x)-w_\pm(y)|^qK_{tq}(x,y)\, dxdy\\
&\quad+\int_{B_\rho}w_\pm(x)\left(\int_{\mathbb{R}^n}\frac{w_\mp^{p-1}(y)}{|x-y|^{n+sp}}+a(x,y)\frac{w_\mp^{q-1}(y)}{|x-y|^{n+tq}}\,dy\right)\,dx\\
&\le C\Bigg[\int_{B_r}\int_{B_r}\frac{|w_\pm(x)+w_\pm(y)|^p}{|x-y|^{n+sp}}\frac{|x-y|^p}{(r-\rho)^p}+a(x,y)\frac{|w_\pm(x)+w_\pm(y)|^q}{|x-y|^{n+tq}}\frac{|x-y|^q}{(r-\rho)^q}\,dxdy\\
&\qquad+\left(\frac{r}{r-\rho}\right)^{n+sp}\int_{B_r}w_\pm(y)\int_{\mathbb{R}^n\setminus B_\rho}\frac{w^{p-1}_\pm(x)}{|x-x_0|^{n+sp}}\,dxdy\\
&\qquad+\left(\frac{r}{r-\rho}\right)^{n+tq}\int_{B_r}w_\pm(y)\int_{\mathbb{R}^n\setminus B_\rho}a(x,y)\frac{w^{q-1}_\pm(x)}{|x-x_0|^{n+tq}}\,dxdy\Bigg],
\end{align*}
where $w_\pm(x):=(u-k)_\pm(x)$ with $k\in\mathbb{R}$, and $C\ge1$ depends on $n,p,q,s,t,\Lambda$.
\end{lemma}
%\left(\frac{r}{r-\rho}\right)^{n+sp}\int_{B_r}w_\pm(y)\left(\int_{\mathbb{R}^n\setminus B_\rho}\frac{w^{p-1}_\pm(x)}{|x-x_0|^{n+sp}}+a(x,y)\left(\frac{r}{r-\rho}\right)^{tq-sp}\frac{w^{q-1}_\pm(x)}{|x-x_0|^{n+tq}}\,dx\right)dy
\begin{proof}
We just provide the sketch of proof, since that is analogous to \cite[Proposition 7.5]{Coz} or \cite[Lemma 3.1]{DFZ}. Here we show this estimate for $w_+$. And the case for $w_-$ is similar. Let $\rho\le\sigma<\tau\le r$, and a cut-off function $\varphi\in C^\infty_0\Big(B_{\frac{\sigma+\tau}{2}}\Big)$ fulfill $0\le\varphi\le1$, $\varphi\equiv1$ in $B_\sigma$ and $|D\varphi|\le \frac{4}{\tau-\sigma}$. Because $u$ is a minimizer, by taking a test function $v:=u-\varphi w_+$, we have
\begin{align*}
0&\le\iint_{\mathcal{C}_{\Omega}} \big(|v(x)-v(y)|^p-|u(x)-u(y)|^p\big)K_{sp}(x,y)\\
&\qquad+a(x,y)\big(|v(x)-v(y)|^q-|u(x)-u(y)|^q\big)K_{tq}(x,y)\,dxdy\\
&=\int_{B_\tau}\int_{B_\tau}\big(|v(x)-v(y)|^p-|u(x)-u(y)|^p\big)K_{sp}(x,y)\,dxdy\\
&\quad+2\int_{B_\tau}\int_{\mathbb{R}^n\setminus B_\tau}\big(|v(x)-v(y)|^p-|u(x)-u(y)|^p\big)K_{sp}(x,y)\,dxdy\\
&\quad+\int_{B_\tau}\int_{B_\tau}a(x,y)\big(|v(x)-v(y)|^q-|u(x)-u(y)|^q\big)K_{tq}(x,y)\,dxdy\\
&\quad+2\int_{B_\tau}\int_{\mathbb{R}^n\setminus B_\tau}a(x,y)\big(|v(x)-v(y)|^q-|u(x)-u(y)|^q\big)K_{tq}(x,y)\,dxdy\\
&=:I_1+I_2+I_3+I_4.
\end{align*}
For $I_3+I_4$, although there is a variable coefficient $a(x,y)$, we could also follow the calculations in \cite[pages 4819--4821]{Coz} and then get
\begin{align*}
I_3+I_4& \leq-\frac{1}{C} \int_{B_{\sigma}} \int_{B_{\sigma}}a(x,y)|w_{+}(x)-w_{+}(y)|^qK_{tq}(x,y)\,dxdy\\
&\quad-\frac{1}{C} \int_{B_\sigma} w_{+}(x)\left(\int_{\Bbb{R}^n} a(x,y)\frac{w^{q-1}_{-}(y)}{|x-y|^{n+tq}}\,d y\right)\,d x   \\
&\quad+C\iint_{(B_{\tau}\times B_\tau)\backslash (B_\sigma\times B_\sigma)}a(x,y)\left|w_{+}(x)-w_{+}(y)\right|^qK_{tq}(x,y)\, dxdy   \\
&\quad+\frac{C}{(\tau-\sigma)^q}\int_{B_\tau}\int_{B_\tau}a(x,y)\frac{|w_+(x)+w_+(y)|^q}{|x-y|^{n+(t-1)q}}\,dxdy   \\
&\quad+\frac{C r^{n+tq}}{(\tau-\sigma)^{n+tq}}
 \int_{B_{r}} \int_{\mathbb{R}^n \backslash  B_\rho}a(x,y)\frac{w_{+}^{q-1}(x) w_{+}(y)}{|x-x_0|^{n+t q}} \,d x d y
%\frac{C}{(r_1-\rho_1)^q}\left\|w_{+}\right\|_{L^q\left(B_r\right)}^q\right) .
\end{align*}
with some positive constant $C$ depending on $n,p,q,s,t,\Lambda$. For the estimate on $I_1+I_2$, it is similar to $I_3+I_4$ and we only substitute $a(x,y), q$ and $t$ with 1, $p$ and $s$ respectively.

Combining these estimates above yields that
\begin{align*}
&\quad\int_{B_\sigma} \int_{B_\sigma}|w_+(x)-w_+(y)|^pK_{sp}(x,y)+a(x,y)|w_+(x)-w_+(y)|^qK_{tq}(x,y)\,dxdy\\
&\quad+\int_{B_\sigma} w_+(x)\left(\int_{\Bbb{R}^n} \frac{w^{p-1}_-(y)}{|x-y|^{n+sp}}+a(x,y)\frac{w^{q-1}_-(y)}{|x-y|^{n+tq}}\,d y\right)\,dx \\ &\leq
C\iint_{(B_{\tau}\times B_\tau)\backslash (B_\sigma\times B_\sigma)}\big[|w_+(x)-w_+(y)|^pK_{sp}(x,y)\\ &\qquad\qquad\qquad\qquad\qquad+a(x,y)|w_+(x)-w_+(y)|^qK_{tq}(x,y)\big]\,dxdy   \\
&\quad+\frac{C}{(\tau-\sigma)^p}\int_{B_\tau}\int_{B_\tau}\frac{|w_+(x)+w_+(y)|^p}{|x-y|^{n+(s-1)p}}\,dxdy\\
&\quad+\frac{C}{(\tau-\sigma)^q}\int_{B_\tau}\int_{B_\tau}a(x,y)\frac{|w_+(x)+w_+(y)|^q}{|x-y|^{n+(t-1)q}}\,dxdy   \\
&\quad+\frac{C r^{n+sp}}{(\tau-\sigma)^{n+sp}}
 \int_{B_r} \int_{\mathbb{R}^n \backslash  B_\rho}\frac{w_+^{p-1}(x) w_+(y)}{|x-x_0|^{n+sp}} \,d x d y\\
&\quad+\frac{C r^{n+tq}}{(\tau-\sigma)^{n+tq}}
 \int_{B_r} \int_{\mathbb{R}^n \backslash  B_\rho}a(x,y)\frac{w_{+}^{q-1}(x) w_{+}(y)}{|x-x_0|^{n+t q}} \,dxdy.
\end{align*}
Now set
\begin{align*}
\Phi(\iota)= & \int_{B_\iota} \int_{B_\iota}|w_+(x)-w_+(y)|^pK_{sp}(x,y)+a(x,y)|w_+(x)-w_+(y)|^qK_{tq}(x,y)\,dxdy\\
&+\int_{B_\iota} w_+(x)\left(\int_{\Bbb{R}^n }\frac{w^{p-1}_-(y)}{|x-y|^{n+sp}}+a(x,y)\frac{w^{q-1}_-(y)}{|x-y|^{n+tq}}\,dy\right)\,dx,  \quad \iota>0.
\end{align*}
Then we obtain
\begin{align*}
\Phi(\sigma) & \leq C\big(\Phi(\tau)-\Phi(\sigma)\big)+\frac{C}{(\tau-\sigma)^p}\int_{B_r}\int_{B_r}\frac{|w_+(x)+w_+(y)|^p}{|x-y|^{n+(s-1)p}}\,dxdy\\
&\quad+\frac{C}{(\tau-\sigma)^q}\int_{B_r}\int_{B_r}a(x,y)\frac{|w_+(x)+w_+(y)|^q}{|x-y|^{n+(t-1)q}}\,dxdy\\
&\quad+ \frac{Cr^{n+sp}}{(\tau-\sigma)^{n+sp}}
 \int_{B_r} \int_{\mathbb{R}^n \backslash  B_\rho}\frac{w_+^{p-1}(x) w_+(y)}{\left|x-x_0\right|^{n+sp}} \,dxdy\\
& \quad+ \frac{Cr^{n+tq}}{(\tau-\sigma)^{n+tq}}
 \int_{B_r} \int_{\mathbb{R}^n \backslash  B_\rho}a(x,y)\frac{w_+^{q-1}(x) w_+(y)}{\left|x-x_0\right|^{n+tq}} \,dxdy
\end{align*}
Applying the technical lemma, Lemma \ref{iteration}, deduces the desired result:
\begin{align*}
\Phi(\rho) & \leq C\Bigg[\frac{1}{(r-\rho)^p}\int_{B_r}\int_{B_r}\frac{|w_+(x)+w_+(y)|^p}{|x-y|^{n+(s-1)p}}\,dxdy\\
&\quad+\frac{1}{(r-\rho)^q}\int_{B_r}\int_{B_r}a(x,y)\frac{|w_+(x)+w_+(y)|^q}{|x-y|^{n+(t-1)q}}\,dxdy\\
&\quad+ \frac{r^{n+sp}}{(r-\rho)^{n+sp}}
 \int_{B_r} \int_{\mathbb{R}^n \backslash  B_\rho}\frac{w_+^{p-1}(x) w_+(y)}{|x-x_0|^{n+sp}} \,dxdy\\
& \quad+ \frac{r^{n+tq}}{(r-\rho)^{n+tq}}
 \int_{B_r} \int_{\mathbb{R}^n \backslash  B_\rho}a(x,y)\frac{w_+^{q-1}(x) w_+(y)}{|x-x_0|^{n+tq}} \,dxdy \Bigg].
\end{align*}
\end{proof}

In what follows, we further derive an extremely significant energy estimate, improving the Caccioppoli-type inequality in Lemma \ref{lem3-1}, which will be utilized over and over again.

\begin{lemma}
\label{lem3-2}
Let $B_R:=B_R(x_0)\subset\Omega$ with $R\le1$, and $u\in\mathcal{A}(\Omega)\cap L^{p-1}_{sp}(\mathbb{R}^n)\cap L^{q-1}_{a,tq}(\Omega,\mathbb{R}^n)$ be a minimizer of the functional \eqref{main}. If $sp\le n$, we suppose $u$ is bounded in $B_R$. Then for $w_\pm(x):=(u-k)_\pm(x)$ with $|k|\le \|u\|_{L^\infty(B_R)}$, we have
\begin{align*}
&\quad\int_{B_\rho}\int_{B_\rho}\frac{|w_\pm(x)-w_\pm(y)|^p}{|x-y|^{n+sp}}+a^-_R\frac{|w_\pm(x)-w_\pm(y)|^q}{|x-y|^{n+tq}}\, dxdy\\
&\quad+\int_{B_\rho}w_\pm(x)\left(\int_{\mathbb{R}^n}\frac{w_\mp^{p-1}(y)}{|x-y|^{n+sp}}+a(x,y)\frac{w_\mp^{q-1}(y)}{|x-y|^{n+tq}}\,dy\right)\,dx\\
&\le C\left(\frac{r}{r-\rho}\right)^{n+q}\Bigg[\int_{B_r}\frac{w_\pm(x)^p}{r^{sp}}+a^-_R\frac{w_\pm^q}{r^{tq}}\,dxdy\\
&\quad+\int_{B_r}w_\pm(y)\left(\int_{\mathbb{R}^n\setminus B_\rho}\frac{w^{p-1}_\pm(x)}{|x-x_0|^{n+sp}}+a(x,y)\frac{w^{q-1}_\pm(x)}{|x-x_0|^{n+tq}}\,dx\right)\,dy\Bigg]
\end{align*}
for $\frac{R}{2}\le \rho<r\le R$, under the following conditions that the coefficient $a$ satisfies \eqref{a} and the growth exponents fulfill
\begin{equation}
\label{exp}
\begin{cases}
tq\le sp+\alpha,  \quad &\text{for } sp\le n,\\[2mm]
tq\le n+\alpha+\left(s-\frac{n}{p}\right)q,  \quad &\text{for } sp>n.
\end{cases}
\end{equation}
Here the constant $C\ge1$ depends on $\textbf{data}(B_R)$
%\begin{equation*}
%data(B_R):=\begin{cases}
%n,p,q,s,t,\Lambda,\alpha,[a]_\alpha,\|u\|_{L^\infty(B_R)},   \quad &\text{if } sp\le n,\\[2mm]
%n,p,q,s,t,\Lambda,\alpha,[a]_\alpha,[u]_{W^{s,p}(B_R)},   \quad &\text{if } sp> n,
%\end{cases}
%\end{equation*}
and the definition of $a^-_R$ is given as \eqref{ar}.
%$$
%a^-_R:=\inf_{B_R\times B_R} a(x,y).
%$$
\end{lemma}

\begin{proof}
Via the definition of $a^-_R$ and Lemma \ref{lem3-1}, it directly follows that
\begin{align}
\label{3-2-1}
&\quad\int_{B_\rho}\int_{B_\rho}\frac{|w_\pm(x)-w_\pm(y)|^p}{|x-y|^{n+sp}}+a^-_R\frac{|w_\pm(x)-w_\pm(y)|^q}{|x-y|^{n+tq}}\, dxdy \nonumber\\
&\quad+\int_{B_\rho}w_\pm(x)\left(\int_{\mathbb{R}^n}\frac{w_\mp^{p-1}(y)}{|x-y|^{n+sp}}+a(x,y)\frac{w_\mp^{q-1}(y)}{|x-y|^{n+tq}}\,dy\right)\,dx \nonumber\\
&\le C\Bigg[\frac{1}{(r-\rho)^p}\int_{B_r}\int_{B_r}\frac{|w_\pm(x)+w_\pm(y)|^p}{|x-y|^{n+(s-1)p}}\,dxdy \nonumber\\
&\qquad+\frac{1}{(r-\rho)^q}\int_{B_r}\int_{B_r}a(x,y)\frac{|w_\pm(x)+w_\pm(y)|^q}{|x-y|^{n+(t-1)q}}\,dxdy \nonumber\\
&\qquad+\left(\frac{r}{r-\rho}\right)^{n+sp}\int_{B_r}w_\pm(y)\int_{\mathbb{R}^n\setminus B_\rho}\frac{w^{p-1}_\pm(x)}{|x-x_0|^{n+sp}}\,dxdy  \nonumber\\
&\qquad+\left(\frac{r}{r-\rho}\right)^{n+tq}\int_{B_r}w_\pm(y)\int_{\mathbb{R}^n\setminus B_\rho}a(x,y)\frac{w^{q-1}_\pm(x)}{|x-x_0|^{n+tq}}\,dxdy\Bigg].
\end{align}
For simplicity, we denote the first two integrals at the right-hand side by $I_1$ and $I_2$ in turn. Now we shall evaluate the integral $I_2$. By \eqref{a}, for $(x,y)\in B_r\times B_r$,
$$
a(x,y)=a(x,y)-a^-_R+a^-_R\le C[a]_\alpha R^\alpha+a^-_R\le C[a]_\alpha r^\alpha+a^-_R,
$$
where we used the fact $\frac{R}{2}\le r\le R$. Then,
\begin{align*}
I_2&=\frac{1}{(r-\rho)^q}\int_{B_r}\int_{B_r}a(x,y)\frac{|w_\pm(x)+w_\pm(y)|^q}{|x-y|^{n+(t-1)q}}\,dxdy \nonumber\\
&\le\frac{a^-_R}{(r-\rho)^q}\int_{B_r}\int_{B_r}\frac{|w_\pm(x)+w_\pm(y)|^q}{|x-y|^{n+(t-1)q}}\,dxdy+\frac{Cr^\alpha}{(r-\rho)^q}\int_{B_r}\int_{B_r}\frac{|w_\pm(x)+w_\pm(y)|^q}{|x-y|^{n+(t-1)q}}\,dxdy\\ \nonumber
&\le \frac{Ca^-_Rr^{(1-t)q}}{(r-\rho)^q}\int_{B_r}w^q_\pm\,dx+\frac{Cr^{\alpha+(1-t)q}}{(r-\rho)^q}\int_{B_r}w^q_\pm\,dx.
\end{align*}
Similarly, it is easy to get
\begin{equation}
\label{3-2-2}
I_1\le \frac{Cr^{(1-s)p}}{(r-\rho)^p}\int_{B_r}w^p_\pm\,dx=\frac{r^p}{(r-\rho)^p}\int_{B_r}\frac{w^p_\pm}{r^{sp}}\,dx.
\end{equation}
If $sp\le n$, through the assumptions that $|k|\le \|u\|_{L^\infty(B_R)}$ and $\eqref{exp}_1$, we deduce
\begin{align*}
\frac{r^{\alpha+(1-t)q}}{(r-\rho)^q}\int_{B_r}w^q_\pm\,dx&=\frac{r^{\alpha+(1-t)q}}{(r-\rho)^q}\int_{B_r}w^{q-p}_\pm w^p_\pm\,dx\\
&\le C\|u\|^{q-p}_{L^\infty(B_r)}\frac{r^{\alpha+(1-t)q+sp}}{(r-\rho)^q}\int_{B_r}\frac{w^p_\pm}{r^{sp}}\,dx\\
&\le C\|u\|^{q-p}_{L^\infty(B_r)}\frac{r^{q}}{(r-\rho)^q}\int_{B_r}\frac{w^p_\pm}{r^{sp}}\,dx.
\end{align*}
If $sp>n$, noting that $W^{s,p}(B_R)\hookrightarrow C^{0,s-\frac{n}{p}}(B_R)$ and $|k|\le \|u\|_{L^\infty(B_R)}$, there holds that
\begin{align*}
\left(\mathop{\rm osc}\limits_{B_R}u\right)^{q-p}&=\left(\frac{\mathop{\rm osc}\limits_{B_R}u}{R^{s-\frac{n}{p}}}\right)^{q-p}R^{\left(s-\frac{n}{p}\right) (q-p)}\le C[u]^{q-p}_{W^{s,p}(B_R)}R^{\left(s-\frac{n}{p}\right) (q-p)}\\
&\le C[u]^{q-p}_{W^{s,p}(B_R)}r^{\left(s-\frac{n}{p}\right) q-(sp-n)}\le C[u]^{q-p}_{W^{s,p}(B_R)}r^{tq-sp-\alpha},
\end{align*}
and further
\begin{align*}
\frac{r^{\alpha+(1-t)q}}{(r-\rho)^q}\int_{B_r}w^q_\pm\,dx&\le\frac{r^{\alpha+(1-t)q}}{(r-\rho)^q}\int_{B_r}\left(\mathop{osc}\limits_{B_R}u\right)^{q-p}w^p_\pm\,dx\\
&\le C[u]^{q-p}_{W^{s,p}(B_R)}\frac{r^{\alpha+(1-t)q+tq-sp-\alpha}}{(r-\rho)^q}\int_{B_r}w^p_\pm\,dx\\
&=C[u]^{q-p}_{W^{s,p}(B_R)}\frac{r^q}{(r-\rho)^q}\int_{B_r}\frac{w^p_\pm}{r^{sp}}\,dx.
\end{align*}
Here we have utilized the conditions that $\frac{R}{2}\le r\le R\le1$ and $tq\le \left(s-\frac{n}{p}\right)q+n+\alpha$. As a result,
\begin{equation}
\label{3-2-3}
I_2\le C\frac{r^q}{(r-\rho)^q}\int_{B_r}\frac{w^p_\pm}{r^{sp}}+a^-_R\frac{w^q_\pm}{r^{tq}}\,dx.
\end{equation}
Merging the displays \eqref{3-2-2}, \eqref{3-2-3} with \eqref{3-2-1} leads to
\begin{align*}
&\quad\int_{B_\rho}\int_{B_\rho}\frac{|w_\pm(x)-w_\pm(y)|^p}{|x-y|^{n+sp}}+a^-_R\frac{|w_\pm(x)-w_\pm(y)|^q}{|x-y|^{n+tq}}\, dxdy\\
&\quad+\int_{B_\rho}w_\pm(x)\left(\int_{\mathbb{R}^n}\frac{w_\mp^{p-1}(y)}{|x-y|^{n+sp}}+a(x,y)\frac{w_\mp^{q-1}(y)}{|x-y|^{n+tq}}\,dy\right)\,dx\\
&\le C\Bigg[\frac{r^q}{(r-\rho)^q}\int_{B_r}\frac{w^p_\pm}{r^{sp}}+a^-_R\frac{w^q_\pm}{r^{tq}}\,dx+\left(\frac{r}{r-\rho}\right)^{n+sp}\int_{B_r}w_\pm(y)\int_{\mathbb{R}^n\setminus B_\rho}\frac{w^{p-1}_\pm(x)}{|x-x_0|^{n+sp}}\,dxdy\\
&\qquad+\left(\frac{r}{r-\rho}\right)^{n+tq}\int_{B_r}w_\pm(y)\int_{\mathbb{R}^n\setminus B_\rho}a(x,y)\frac{w^{q-1}_\pm(x)}{|x-x_0|^{n+tq}}\,dxdy\Bigg],
\end{align*}
which implies the desired result.
\end{proof}

\section{Towards weak Harnack inequalities: expansion of positivity}
\label{sec4}

We in this part collect some measure theoretical results as necessary ingredients for proving Theorems \ref{Thm1}--\ref{Thm3}. The forthcoming lemma studies the measure shrinking features on the minimizers of the functional \eqref{main}, which is the fundamental step in investigating weak Harnack estimates.

\begin{lemma}
\label{lem4-1}
Let $u\in\mathcal{A}(\Omega)\cap L^{p-1}_{sp}(\mathbb{R}^n)\cap L^{q-1}_{a,tq}(\Omega,\mathbb{R}^n)$, nonnegative in a ball $B_{4R}:=B_{4R}(x_0)\subset \Omega$ with $R\le1$, be a minimizer of \eqref{main}. Assume $u$ is bounded in $B_{4R}$ if $sp\le n$. Suppose also that
\begin{equation}
\label{4-1-1}
|B_{2R}\cap\{u\ge \tau\}|\ge \nu|B_{2R}|
\end{equation}
for $\tau\ge0$ and some $\nu\in(0,1]$. Then under the hypotheses \eqref{a} and \eqref{thm1}, there holds that, for each $\delta\in\left(0,\frac{1}{8}\right]$, either
\begin{equation}
\label{4-1-2}
\mathrm{Tail} (u_-;x_0,4R)+\mathrm{Tail}_a(u_-;x_0,4R)\ge h_{4R}(\delta \tau)
\end{equation}
or
$$
|B_{2R}\cap\{u\le 2\delta \tau\}|\le C\frac{\delta^{p-1}}{\nu}|B_{2R}|,
$$
where the positive constant $C$ depends only upon $\textbf{data}(B_{4R})$. %$n,p,q,s,t,\Lambda,\alpha,[a]_\alpha,\|u\|_{L^\infty(B_{4R})}$.
Here the definition of $h_{4R}(\delta \tau)$ can be found in \eqref{hr}.
%$$
%h_{4R}(\delta \tau):=\frac{(\delta \tau)^{p-1}}{(4R)^{sp}}+a^-_{4R}\frac{(\delta \tau)^{q-1}}{(4R)^{tq}}.
%$$
%and
%$$
%\mathrm{Tail}_a(u_-;x_0,4R):=\int_{\mathbb{R}^n\setminus B_{4R}}a(x,y)\frac{u_-^{q-1}(x)}{|x-x_0|^{n+tq}}\,dx.
%$$
\end{lemma}

\begin{proof}
Assume \eqref{4-1-2} is not true. Then let us show the measure theoretical result. We can assume $\tau\le \|u\|_{L^\infty(B_{4R})}$ since \eqref{4-1-1} and $u$ is bounded. %Set
%$$
%h_r(t)=\frac{t^{p-1}}{r^{sp}}+a^-_r\frac{t^{q-1}}{r^{tq}}
%$$
%and
Fix $l\in\left[\frac{\delta \tau}{2},\|u\|_{L^\infty(B_{4R})}\right]$. Thanks to Lemma \ref{lem3-2} with $\rho:=2R$ and $r:=4R$, we have
\begin{align}
\label{4-1-3}
&\quad\int_{B_{2R}}w_-(x)\left(\int_{\mathbb{R}^n}\frac{w_+^{p-1}(y)}{|x-y|^{n+sp}}+a(x,y)\frac{w_+^{q-1}(y)}{|x-y|^{n+tq}}\,dy\right)\,dx \nonumber\\
&\le C\Bigg[\int_{B_{4R}}\frac{w_-^p(x)}{{(4R)}^{sp}}+a^-_{4R}\frac{w_-^q(x)}{{(4R)}^{tq}}\,dx   \nonumber\\
&\qquad+\int_{B_{4R}}w_-(y)\left(\int_{\mathbb{R}^n\setminus B_{2R}}\frac{w^{p-1}_-(x)}{|x-x_0|^{n+sp}}+a(x,y)\frac{w^{q-1}_-(x)}{|x-x_0|^{n+tq}}\,dx\right)\,dy\Bigg]   \nonumber\\
&\le C\Bigg[H_{4R}(l)|B_{4R}|+\int_{B_{4R}}w_-(y)\int_{\mathbb{R}^n\setminus B_{2R}}\frac{w^{p-1}_-(x)}{|x-x_0|^{n+sp}}\,dxdy  \nonumber\\
&\qquad+\int_{B_{4R}}w_-(y)\int_{\mathbb{R}^n\setminus B_{2R}}a(x,y)\frac{w^{q-1}_-(x)}{|x-x_0|^{n+tq}}\,dxdy\Bigg]  \nonumber\\
&=:C(H_{4R}(l)|B_{4R}|+J_1+J_2),
\end{align}
where $w_\pm:=(u-l)_\pm$, the fact $u\ge0$ in $B_{4R}$ is used and $H_{4R}(l)$ is defined in \eqref{Hr}. Now consider the nonlocal integrals $J_1$ and $J_2$. Recalling the nonnegativity of $u$ in $B_{4R}$ again, we can see
\begin{align*}
J_2&\le l\int_{B_{4R}}\int_{\mathbb{R}^n\setminus B_{2R}}a(x,y)\frac{(u-l)^{q-1}_-(x)}{|x-x_0|^{n+tq}}\,dxdy\\
&\le Cl \left(\int_{B_{4R}}\int_{\mathbb{R}^n\setminus B_{4R}}a(x,y)\frac{u^{q-1}_-(x)}{|x-x_0|^{n+tq}}\,dxdy+\int_{B_{4R}}\int_{\mathbb{R}^n\setminus B_{2R}}a(x,y)\frac{l^{q-1}}{|x-x_0|^{n+tq}}\,dxdy\right) \\
&=:Cl(J_{21}+J_{22}).
\end{align*}
For $J_{22}$, from the condition \eqref{a} and the definition of $a^-_{4R}$, it follows that
\begin{align*}
a(x,y)&=a(x,y)-a(y,y)+a(y,y)-a^-_{4R}+a^-_{4R}\\
&\le [a]_\alpha|x-y|^\alpha+C[a]_\alpha R^\alpha+a^-_{4R}\\
&\le [a]_\alpha(|x-x_0|+|x_0-y|)^\alpha+C[a]_\alpha R^\alpha+a^-_{4R}\\
&\le 4[a]_\alpha|x-x_0|^\alpha+C[a]_\alpha R^\alpha+a^-_{4R}
\end{align*}
for $x\in\mathbb{R}^n\setminus B_{2R}$ and $y\in B_{4R}$. Therefore, in view of $R\le1$ and $\alpha<tq\le sp+\alpha$, we arrive at
\begin{align*}
J_{22}&\le Cl^{q-1}|B_{4R}|\int_{\mathbb{R}^n\setminus B_{2R}}\frac{1}{|x-x_0|^{n+tq-\alpha}}+\frac{R^\alpha+a^-_{4R}}{|x-x_0|^{n+tq}}\,dx\\
&\le Cl^{q-1}|B_{4R}|(R^{\alpha-tq}+a^-_{4R}R^{-tq})\\
&\le C|B_{4R}|\left(\frac{l^{q-p}l^{p-1}}{R^{sp}}+a^-_{4R}\frac{l^{q-1}}{R^{tq}}\right)\\
&\le Ch_{4R}(l)|B_{4R}|,
\end{align*}
where we utilized $l\le\|u\|_{L^\infty(B_{4R})}$ and $C$ depends on $\textbf{data}(B_{4R})$.  When $sp>n$, we observe the relation $sp+\alpha\le n+\alpha+\left(s-\frac{n}{p}\right)q$. %$n,q,t,\alpha$, $[a]_\alpha$, $\|u\|_{L^\infty(B_{4R})}$.
Here we know the requirement $\alpha<tq$ is exploited to assure the improper integral $\int_{\mathbb{R}^n\setminus B_{2R}}|x-x_0|^{-(n+tq-\alpha)}\,dx$ makes sense. As for $J_{21}$, applying the reverse of \eqref{4-1-2} immediately infers
\begin{align*}
J_{21}\le |B_{4R}|\sup_{y\in B_{4R}}\int_{\mathbb{R}^n\setminus B_{4R}}a(x,y)\frac{u_-^{q-1}(x)}{|x-x_0|^{n+tq}}\,dx\le h_{4R}(l)|B_{4R}|.
\end{align*}
That is,
\begin{equation}
\label{4-1-4}
J_2\le CH_{4R}(l)|B_{4R}|.
\end{equation}
Similarly, by the reverse of \eqref{4-1-2} once more,
\begin{align}
\label{4-1-5}
J_1&\le Cl|B_{4R}|(\mathrm{Tail}(u_-;x_0,4R)+l^{p-1}R^{-sp}) \nonumber\\
&\le CH_{4R}(l)|B_{4R}|.
\end{align}
Putting together \eqref{4-1-4} and \eqref{4-1-5} with \eqref{4-1-3}, we derive
\begin{align}
\label{4-1-6}
&\quad\int_{B_{2R}}\int_{B_{2R}}\frac{w_-(x)w_+^{p-1}(y)}{|x-y|^{n+sp}}\,dydx+a^-_{4R}\int_{B_{2R}}\int_{B_{2R}}\frac{w_-(x)w_+^{q-1}(y)}{|x-y|^{n+tq}}\,dydx \nonumber\\
&\le\int_{B_{2R}}w_-(x)\left(\int_{B_{2R}}\frac{w_+^{p-1}(y)}{|x-y|^{n+sp}}+a(x,y)\frac{w_+^{q-1}(y)}{|x-y|^{n+tq}}\,dy\right)\,dx \nonumber\\
&\le CH_{4R}(l)|B_{4R}|.
\end{align}

In the sequel, we make two mutually exclusive alternatives:
\begin{equation}
\label{4-1-7}
\left(\frac{\delta \tau}{(4R)^{s}}\right)^p\ge a^-_{4R}\left(\frac{\delta \tau}{(4R)^{t}}\right)^q  \quad\text{and} \quad
\left(\frac{\delta \tau}{(4R)^{s}}\right)^p< a^-_{4R}\left(\frac{\delta \tau}{(4R)^{t}}\right)^q.
\end{equation}

\medskip

\textbf{Case $\eqref{4-1-7}_1$}. By letting $l=4\delta \tau$, we conclude from $\delta\in\left(0,\frac{1}{8}\right]$, \eqref{4-1-1} and \eqref{4-1-6} that
\begin{align*}
C\left(\frac{\delta \tau}{(4R)^{s}}\right)^p|B_{4R}|&\ge\int_{B_{2R}}\int_{B_{2R}}\frac{(u-4\delta\tau)_-(x)(u-4\delta\tau)_+^{p-1}(y)}{|x-y|^{n+sp}}\,dydx \nonumber\\
&\ge (4R)^{-n-sp}\int_{B_{2R}\cap\{u\ge \tau\}}(u-4\delta\tau)_+^{p-1}\,dy\cdot\int_{B_{2R}\cap\{u\le 2\delta\tau\}}(u-4\delta\tau)_-\,dx \nonumber\\
&\ge(4R)^{-n-sp}[(1-4\delta)\tau]^{p-1}|B_{2R}\cap\{u\ge \tau\}|\cdot 2\delta \tau|B_{2R}\cap\{u\le 2\delta\tau\}| \nonumber\\
&\ge \frac{\nu \delta \tau^p}{2^{p-2}(4R)^{n+sp}}|B_{2R}||B_{2R}\cap\{u\le 2\delta\tau\}|,
\end{align*}
i.e.,
\begin{equation}
\label{4-1-8}
|B_{2R}\cap\{u\le 2\delta\tau\}|\le C\frac{\delta^{p-1}}{\nu}|B_{2R}|.
\end{equation}

\medskip

\textbf{Case $\eqref{4-1-7}_2$}. Notice $a^-_{4R}>0$ at this time. Analogously, it yields that
\begin{align*}
Ca^-_{4R}\left(\frac{\delta \tau}{(4R)^{t}}\right)^q|B_{4R}|&\ge a^-_{4R}\int_{B_{2R}}\int_{B_{2R}}\frac{(u-4\delta\tau)_-(x)(u-4\delta\tau)_+^{q-1}(y)}{|x-y|^{n+tq}}\,dydx\\
&\ge a^-_{4R}\frac{\nu \delta \tau^q}{2^{q-2}(4R)^{n+tq}}|B_{2R}||B_{2R}\cap\{u\le 2\delta\tau\}|,
\end{align*}
namely,
\begin{equation}
\label{4-1-9}
|B_{2R}\cap\{u\le 2\delta\tau\}|\le C\frac{\delta^{q-1}}{\nu}|B_{2R}|.
\end{equation}
In summary, we deduce the desired conclusion from \eqref{4-1-8} and \eqref{4-1-9} due to $\delta^{q-1}\le\delta^{p-1}$.
\end{proof}

The second step towards weak Harnack inequalities is made in the upcoming growth lemma, a certain pointwise estimate extracted through some critical measure density condition in a ball.

\begin{lemma}
\label{lem4-2}
Let $u\in\mathcal{A}(\Omega)\cap L^{p-1}_{sp}(\mathbb{R}^n)\cap L^{q-1}_{a,tq}(\mathbb{R}^n)$, nonnegative in a ball $B_{4R}:=B_{4R}(x_0)\subset \Omega$ with $R\le1$, be a minimizer of \eqref{main}. For the case $sp\le n$, assume $u$ is bounded in $B_{4R}$. Besides, suppose that
\begin{equation}
\label{4-2-1}
|B_{2R}\cap\{u\ge \tau\}|\ge \nu|B_{2R}|
\end{equation}
for $\tau\ge0$ and some $\nu\in(0,1]$. Then under the assumptions \eqref{a} and \eqref{thm1}, there is a $\delta\in\left(0,\frac{1}{8}\right]$ that depends on $\textbf{data}(B_{4R})$ and $\nu$ such that either
\begin{equation}
\label{4-2-2}
\mathrm{Tail} (u_-;x_0,4R)+\mathrm{Tail}_a(u_-;x_0,4R)\ge h_{4R}(\delta \tau)
\end{equation}
or
$$
u\ge \delta \tau  \quad\text{in } B_R.
$$
\end{lemma}

\begin{proof}
If \eqref{4-2-2} is not true, we are going to prove $u\ge \delta \tau$ in $B_R$ in the remaining contents. Here assume $\tau\le\|u\|_{L^\infty}(B_{4R})$ from \eqref{4-2-1} and the boundedness of $u$. Let $\delta \tau\le l<k\le2\delta\tau$ and $R\le\rho<r\le 2R$. Define
$$
E^-(k,r):=B_r\cap\{u\le k\}.
$$
We pick
\begin{equation}
\label{4-2-3}
\beta:=\frac{p_s^*}{p}\le\frac{q_t^*}{q}.
\end{equation}
In fact, if $tq<n$, then $\beta=\frac{n}{n-sp}\le\frac{n}{n-tq}$ via $sp\le tq$; when $tq\ge n$, then we take $q_t^*$ sufficiently large so that \eqref{4-2-3} holds true.

In the scenario $a^-_{4R}>0$, we can see that the minimizer $u$ belongs to $W^{s,p}(B_{4R})\cap W^{t,q}(B_{4R})$ so we can employ the fractional Sobolev embedding theorem and H\"{o}lder inequality to deduce
\begin{align*}
&\quad\left[\left(\frac{k-l}{\rho^s}\right)^p+a^-_{4R}\left(\frac{k-l}{\rho^t}\right)^q\right]\left(\frac{E^-(l,\rho)}{|B_\rho|}\right)^\frac{1}{\beta}\\
&\le\left(\mint_{B_\rho}\left[\left(\frac{(u-k)_-}{\rho^s}\right)^p+a^-_{4R}\left(\frac{(u-k)_-}{\rho^t}\right)^q\right]^\beta\,dx\right)^\frac{1}{\beta}\\
&\le C\left(\mint_{B_\rho}\left(\frac{w_--(w_-)_{B_\rho}}{\rho^s}\right)^{p\beta}\,dx\right)^\frac{1}{\beta}+Ca^-_{4R}\left(\mint_{B_\rho}
\left(\frac{w_--(w_-)_{B_\rho}}{\rho^t}\right)^{q\beta}\,dx\right)^\frac{1}{\beta}\\
&\quad+C\mint_{B_\rho}\left(\frac{w_-}{\rho^s}\right)^p+a^-_{4R}\left(\frac{w_-}{\rho^t}\right)^{q}\,dx\\
&\le C\mint_{B_\rho}\int_{B_\rho}\frac{|w_-(x)-w_-(y)|^p}{|x-y|^{n+sp}}\,dxdy+C\mint_{B_\rho}\int_{B_\rho}a^-_{4R}\frac{|w_-(x)-w_-(y)|^q}{|x-y|^{n+tq}}\,dxdy\\
&\quad+CH_{4R}(k)\frac{|E^-(k,r)|}{|B_r|},
\end{align*}
where $w_-:=(u-k)_-$, and in the last inequality the nonnegativity of $u$ in $B_{4R}$ and the relation $R\le \rho<r\le 2R$ were used. Exploiting Lemma \ref{lem3-2} in the display above and recalling the rang of $\rho$ and $r$, we arrive at
\begin{align*}
&\quad H_{4R}(k-l)\left(\frac{|E^-(l,\rho)|}{|B_\rho|}\right)^\frac{1}{\beta}\\
&\le C\left(\frac{r}{r-\rho}\right)^{n+q}\Bigg[\mint_{B_r}\left(\frac{w_-}{\rho^s}\right)^p+a^-_{4R}\left(\frac{w_-}{\rho^t}\right)^{q}\,dx\\
&\quad+\mint_{B_r}w_-(y)\left(\int_{\mathbb{R}^n\setminus B_\rho}\frac{w_-^{p-1}(x)}{|x-x_0|^{n+sp}}+a(x,y)\frac{w_-^{q-1}(x)}{|x-x_0|^{n+tq}}\,dx\right)dy\Bigg]+CH_{4R}(k)\frac{|E^-(k,r)|}{|B_r|}\\
&\le CH_{4R}(k)\frac{|E^-(k,r)|}{|B_r|}+C\left(\frac{r}{r-\rho}\right)^{n+q}H_{4R}(k)\frac{|E^-(k,r)|}{|B_r|}\\
&\quad+C\left(\frac{r}{r-\rho}\right)^{n+q}\mint_{B_r}w_-(y)\left(\int_{\mathbb{R}^n\setminus B_\rho}\frac{w_-^{p-1}(x)}{|x-x_0|^{n+sp}}+a(x,y)\frac{w_-^{q-1}(x)}{|x-x_0|^{n+tq}}\,dx\right)dy,
\end{align*}
where $H_{4R}(k)$ is defined as \eqref{Hr}. Analogous to the treatment of $J_1$ and $J_2$ in the proof of Lemma \ref{lem4-1}, we by the reverse of \eqref{4-2-2} have
\begin{align*}
&\quad\int_{B_r}w_-(y)\int_{\mathbb{R}^n\setminus B_\rho}\frac{w_-^{p-1}(x)}{|x-x_0|^{n+sp}}\,dxdy\\
&\le Ck|E^-(k,r)|(\mathrm{Tail}(u_-;x_0,4R)+k^{p-1}\rho^{-sp}) \le CH_{4R}(k)|E^-(k,r)|
\end{align*}
and
\begin{align*}
&\quad\int_{B_r}w_-(y)\int_{\mathbb{R}^n\setminus B_\rho}a(x,y)\frac{w_-^{q-1}(x)}{|x-x_0|^{n+tq}}\,dxdy\\
&\le Ck\int_{B_r}\chi_{E^-(k,r)}\int_{\mathbb{R}^n\setminus B_{4R}}a(x,y)\frac{u_-^{q-1}(x)}{|x-x_0|^{n+tq}}\,dxdy\\
&\quad+Ck^q\int_{B_r}\chi_{E^-(k,r)}\int_{\mathbb{R}^n\setminus B_\rho}\frac{a(x,y)}{|x-x_0|^{n+tq}}\,dxdy\\
&\le Ck|E^-(k,r)|\mathrm{Tail}_a(u_-;x_0,4R)+Ck^q|E^-(k,r)|(\rho^{\alpha-tq}+a^-_{4R}\rho^{-tq})\\
&\le C\left(H_{4R}(k)+k^{q-p}\frac{k^p}{(4R)^{sp}}+a^-_{4R}\frac{k^q}{(4R)^{tq}}\right)|E^-(k,r)|\\
&\le CH_{4R}(k)|E^-(k,r)|.
\end{align*}
%with the constant $C>0$ depending on $n,p,q,s,t,\Lambda,\alpha,[a]_\alpha$ and $\|u\|_{L^\infty(B_{4R})}$.
Here we made use of $\alpha<tq\le sp+\alpha$, $\rho\approx R\le1$ and $k\le \|u\|_{L^\infty(B_{4R})}$. Merging the preceding three estimates infers
\begin{equation}
\label{4-2-4}
\frac{|E^-(l,\rho)|}{|B_\rho|}\le C\left(\frac{r}{r-\rho}\right)^{(n+q)\beta}\left(\frac{H_{4R}(k)}{H_{4R}(k-l)}\right)^\beta\left(\frac{|E^-(k,r)|}{|B_r|}\right)^\beta
\end{equation}
with $C>1$ depending upon $\textbf{data}(B_{4R})$.

Now we perform an iteration procedure. Denote
$$
r_i=(1+2^{-i})R, \quad k_i=(1+2^{-i})\delta\tau \quad\text{and}\quad  Y_i=\frac{|E^-(k_i,r_i)|}{|B_{r_i}|}
$$
for $i=0,1,2,\cdots$. Applying \eqref{4-2-4} with $r=r_i,\rho=r_{i+1},k=k_i$ and $l=k_{i+1}$ obtains
\begin{equation}
\label{4-2-5}
Y_{i+1}\le C2^{i(n+2q)\beta} Y_i^{(\beta-1)+1},
\end{equation}
where we need to notice
$$
\frac{H_{4R}(k_i)}{H_{4R}(k_i-k_{i+1})}\le C\frac{H_{4R}(\delta\tau)}{H_{4R}(2^{-i-1}\delta\tau)}\le C2^{iq}.
$$
At this point, to utilize the geometric convergence lemma, we enforce
$$
Y_0=\frac{|E^-(2\delta\tau,2R)|}{|B_{2R}|}\le C^\frac{-1}{\beta-1}2^{-\frac{(n+2q)\beta}{(\beta-1)^2}}=:\nu_0.
$$
This can be realized through Lemma \ref{lem4-1}. We choose some $\delta\in\left(0,\frac{1}{8}\right]$, depending on $\textbf{data}(B_{4R})$ and $\nu$, such that
\begin{equation}
\label{4-2-6}
C\frac{\delta^{p-1}}{\nu}\le \nu_0.
\end{equation}
That is to say, we deduce from \eqref{4-2-5}, \eqref{4-2-6} and Lemma \ref{lem-2-3} that
$$
Y_i\rightarrow0  \quad (i\rightarrow\infty),
$$
namely,
$$
u\ge\delta \tau \quad\text{in }  B_R.
$$

If $a^-_{4R}=0$, this case is easier, and we just repeat the previous processes without the terms involving $a^-_{4R}$. Hence, the proof is finished now.
\end{proof}

Almost verbatim following the proof of Lemma \ref{lem4-2}, we can deduce a slightly different De Giorgi-type lemma as below.

\begin{corollary}
\label{cor4-2-1}
Let the assumptions \eqref{a} and \eqref{thm1} hold. Suppose that $u\in\mathcal{A}(\Omega)\cap L^{p-1}_{sp}(\mathbb{R}^n)\cap L^{q-1}_{a,tq}(\Omega,\mathbb{R}^n)$, nonnegative in a ball $B_{4R}:=B_{4R}(x_0)\subset \Omega$ with $R\le1$, is a minimizer of \eqref{main}. If $sp\le n$, assume $u$ is bounded in $B_{4R}$. Then one can find a constant $\nu\in(0,1)$, depending on $\textbf{data}(B_{4R})$, such that if
$$
|B_{2R}\cap\{u\le \tau\}|\le \nu|B_{2R}|
$$
with a parameter $\tau\ge0$, then whenever
$$
\mathrm{Tail} (u_-;x_0,4R)+\mathrm{Tail}_a(u_-;x_0,4R)\le h_{4R}(\tau),
$$
it holds that
$$
u\ge\frac{1}{2}\tau  \quad \text{in } B_R.   % a.e.
$$
\end{corollary}

With the help of a Krylov-Sofonov type covering lemma \cite[Lemma 7.2]{Juha2001} (see also \cite[Lemma 6.6]{Coz}) and Lemma \ref{lem4-2} above, we shall arrive at the forthcoming conclusion refining Lemma \ref{lem4-2}.

\begin{corollary}
\label{cor4-2-2}
Let $k\in \mathbb{N}$ and the conditions \eqref{a}, \eqref{thm1} be fulfilled. Assume that $u\in\mathcal{A}(\Omega)\cap L^{p-1}_{sp}(\mathbb{R}^n)\cap L^{q-1}_{a,tq}(\Omega,\mathbb{R}^n)$ nonnegative in a ball $B_{4R}:=B_{4R}(x_0)\subset \Omega$ with $R\le1$ is a minimizer of \eqref{main}. Suppose also $u$ is bounded in $B_{4R}$ if  $sp\le n$. When
$$
|B_{R}\cap\{u\ge \tau\}|\ge \nu^k|B_{R}|,
$$
for $\tau\ge0$ and $\nu\in(0,1]$, there exists a constant $\delta\in\left(0,\frac{1}{8}\right]$ depending only upon $\textbf{data}(B_{4R})$ and $\nu$ (not upon $k$) such that either
$$
\mathrm{Tail} (u_-;x_0,4R)+\mathrm{Tail}_a(u_-;x_0,4R)\ge h_{4R}(\delta^k\tau),
$$
or
$$
u\ge\delta^k\tau  \quad \text{in } B_R.   %\text{a.e.  in }
$$
\end{corollary}

The proof of this statement is highly analogous to that of \cite[Lemma 6.7]{Coz}, as the structure of the functional \eqref{main} is not exploited at all except for addressing the nonlocal tails in the showing processes. Indeed, there is an extra nonlocal tail with weight, $\mathrm{Tail}_a$, but we can treat it in a similar way to the tail generated by the $p$-growth term. Hence we omit the details here to avoid superfluous repetition.

The measure shrinking estimate aforementioned, Lemma \ref{lem4-1}, needs critical measure density information ahead of time, but we could drop such %the measure theoretical
condition in the next conclusion encoded according to the local integrals now.

\begin{lemma}
\label{lem4-3}
Let the assumptions \eqref{a}, \eqref{thm1} be in force. Let $u\in\mathcal{A}(\Omega)\cap L^{p-1}_{sp}(\mathbb{R}^n)\cap L^{q-1}_{a,tq}(\Omega,\mathbb{R}^n)$, nonnegative in a ball $B_{4R}:=B_{4R}(x_0)\subset \Omega$ with $R\le1$, be a minimizer of \eqref{main}. We assume $u$ is bounded in $B_{4R}$ when $sp\le n$. Then there exists a $C>1$, depending on $\textbf{data}(B_{4R})$, such that whenever
\begin{equation}
\label{4-3-1}
\mathrm{Tail} (u_-;x_0,4R)+\mathrm{Tail}_a(u_-;x_0,4R)\le h_{4R}(\tau)
\end{equation}
it holds that
\begin{equation}
\label{4-3-2}
|\{u\le \tau\}\cap B_{2R}|\le C\left(\frac{\tau^{p-1}}{(u^{p-1})_{B_{2R}}}+\frac{\tau^{q-1}}{(u^{q-1})_{B_{2R}}}\right)|B_{2R}|
\end{equation}
with $0\le \tau$.% and $(v)_{B_r}:=\mint_{B_r}v\,dx$.\le \|u\|_{L^\infty(B_{4R})}
\end{lemma}

\begin{proof}
Set $w_\pm:=(u-2\tau)_\pm$ and suppose $\tau\le \|u\|_{L^\infty(B_{4R})}$ (Otherwise, \eqref{4-3-2} does hold true plainly by the boundedness of $u$ in $B_{4R}$) in this proof. According to Lemma \ref{lem3-2}, we get% and the estimates on $J_1,J_2$, we get
\begin{align}
\label{4-3-3}
&\quad\int_{B_{2R}}w_-(x)\left(\int_{\mathbb{R}^n}\frac{w_+^{p-1}(y)}{|x-y|^{n+sp}}+a(x,y)\frac{w_+^{q-1}(y)}{|x-y|^{n+tq}}\,dy\right)\,dx \nonumber\\
&\le C\Bigg[\int_{B_{4R}}\frac{w_-^p(x)}{(4R)^{sp}}+a^-_{4R}\frac{w_-^q(x)}{(4R)^{tq}}\,dx \nonumber\\
&\quad+\int_{B_{4R}}w_-(y)\int_{\mathbb{R}^n\setminus B_{2R}}\frac{w_-^{p-1}(x)}{|x-x_0|^{n+sp}}+a(x,y)\frac{w_-^{q-1}(x)}{|x-x_0|^{n+tq}}\,dxdy\Bigg]  \nonumber\\
&\le C\left(\frac{\tau^p}{(4R)^{sp}}+a^-_{4R}\frac{\tau^q}{(4R)^{tq}}\right)|B_{4R}|+C\tau|B_{4R}|(\tau^{p-1}(2R)^{-sp}+\mathrm{Tail}(u_-;x_0,4R)) \nonumber\\
&\quad+C\tau|B_{4R}|[\tau^{q-1}(R^{\alpha-tq}+a^-_{4R}R^{-tq})+\mathrm{Tail}_a(u_-;x_0,4R)]  \nonumber\\
&\le CH_{4R}(\tau)|B_{4R}|,
\end{align}
where in the last line we have exploited \eqref{4-3-1}, $R^{\alpha-tq}\le R^{-sp}$ by $tq\le sp+\alpha$ and $\tau^q=\tau^{q-p}\tau^p\le \|u\|_{L^\infty(B_{4R})}^{q-p}\tau^p$ by $\tau\le\|u\|_{L^\infty(B_{4R})}$. Here for the specific evaluations on the integral
$$
\int_{B_{4R}}w_-(y)\int_{\mathbb{R}^n\setminus B_{2R}}\frac{w_-^{p-1}(x)}{|x-x_0|^{n+sp}}+a(x,y)\frac{w_-^{q-1}(x)}{|x-x_0|^{n+tq}}\,dxdy,
$$
one can refer to the integrals $J_1,J_2$ in \eqref{4-1-3}.

In what follows, let us identify two different options:
\begin{equation}
\label{4-3-4}
\left(\frac{\tau}{(4R)^{s}}\right)^p\ge a^-_{4R}\left(\frac{\tau}{(4R)^{t}}\right)^q  \quad\text{and} \quad
\left(\frac{\tau}{(4R)^{s}}\right)^p< a^-_{4R}\left(\frac{\tau}{(4R)^{t}}\right)^q.
\end{equation}

\medskip

\textbf{Case $\eqref{4-3-4}_1$}. It follows from \eqref{4-3-3} that
\begin{equation}
\label{4-3-5}
\int_{B_{2R}}w_-(x)\int_{B_{2R}}\frac{w_+^{p-1}(y)}{|x-y|^{n+sp}}\,dydx\le CH_{4R}(\tau)|B_{4R}|\le C\left(\frac{\tau}{(4R)^{s}}\right)^p|B_{4R}|.
\end{equation}
Now recall the basic inequality $u^{p-1}_+\le C(p)(u-\tau)^{p-1}_++C(p)\tau^{p-1}$ for $\tau>0$. For the left-hand side of \eqref{4-3-5}, we further derive
\begin{align}
\label{4-3-6}
&\quad\int_{B_{2R}}\int_{B_{2R}}\frac{w_-(x)w_+^{p-1}(y)}{|x-y|^{n+sp}}\,dydx  \nonumber\\
&\ge \int_{B_{2R}}\int_{B_{2R}}\frac{w_-(x)\left(\frac{1}{C}u^{p-1}(y)-\tau^{p-1}\right)}{(4R)^{n+sp}}\,dydx  \nonumber\\
&\ge \frac{1}{C(4R)^{sp}}\int_{B_{2R}}w_-(x)\,dx\mint_{B_{2R}}u^{p-1}(y)\,dy-C\left(\frac{\tau}{(4R)^{s}}\right)^p|B_{2R}|,
\end{align}
where in the last line we observe $w_-\le \tau$ for $x\in B_{2R}$ by $u\ge0$ in $B_{4R}$. Combining \eqref{4-3-5} and \eqref{4-3-6} leads to
$$
\int_{B_{2R}}w_-(x)\,dx\mint_{B_{2R}}u^{p-1}(y)\,dy\le C\tau^p|B_{2R}|.
$$
Now let us examine the lower bound on $\int_{B_{2R}}w_-(x)\,dx$,
\begin{equation}
\label{4-3-7}
\int_{B_{2R}}w_-(x)\,dx\ge\int_{B_{2R}\cap\{u\le\tau\}}(u-2\tau)_-(x)\,dx\ge \tau|B_{2R}\cap\{u\le\tau\}|.
\end{equation}
As a result, it holds that
$$
|B_{2R}\cap\{u\le\tau\}|\le C\tau^{p-1}|B_{2R}|\left(\mint_{B_{2R}}u^{p-1}\,dy\right)^{-1}.
$$

\textbf{Case $\eqref{4-3-4}_2$}. In this case, we know $a^-_{4R}>0$. Similar to Case $\eqref{4-3-4}_1$, based on \eqref{4-3-3} we evaluate
\begin{align*}
Ca^-_{4R}\left(\frac{\tau}{(4R)^{t}}\right)^q|B_{4R}|&\ge \int_{B_{2R}}\int_{B_{2R}}a(x,y)\frac{w_-(x)w_+^{q-1}(y)}{|x-y|^{n+tq}}\,dydx\\
&\ge a^-_{4R}\int_{B_{2R}}\int_{B_{2R}}\frac{w_-(x)w_+^{q-1}(y)}{|x-y|^{n+tq}}\,dydx\\
&\ge \frac{a^-_{4R}}{C(4R)^{tq}}\int_{B_{2R}}w_-(x)\,dx\mint_{B_{2R}}u^{q-1}(y)\,dy-C\frac{a^-_{4R}\tau^q}{(4R)^{tq}}|B_{2R}|.
\end{align*}
This inequality together with \eqref{4-3-7} indicates
$$
|B_{2R}\cap\{u\le\tau\}|\le C\tau^{q-1}|B_{2R}|\left(\mint_{B_{2R}}u^{q-1}\,dy\right)^{-1}.
$$
In view of the results in Case $\eqref{4-3-4}_1$ and Case $\eqref{4-3-4}_2$, the measure estimate \eqref{4-3-2} directly follows.
\end{proof}

We end this section by a measure shrinking lemma of diverse types with Lemma \ref{lem4-3}, which checks the effect of the long-range property of minimizers of \eqref{main} (i.e., nonlocal tails). However, what we have to mention is that the modulating coefficient $a(\cdot,\cdot)$ is supposed to carry a positive lower bound instead of $a(\cdot,\cdot)\ge0$. The reason is that we do not know how to estimate the nonlocal tail with weight, $\mathrm{Tail}_a$ in \eqref{4-4-4}, from below up to now, provided $a(\cdot,\cdot)$ is just nonnegative ($a$ may be zero at some points). Here one needs to observe the coefficient remains unbounded from above potentially!

\begin{lemma}
\label{lem4-4}
Let the preconditions that \eqref{a}, \eqref{thm1} and $a(\cdot,\cdot)\ge\lambda>0$ be satisfied. Suppose $u\in\mathcal{A}(\Omega)\cap L^{p-1}_{sp}(\mathbb{R}^n)\cap L^{q-1}_{a,tq}(\Omega,\mathbb{R}^n)$, nonnegative in a ball $B_{4R}:=B_{4R}(x_0)\subset \Omega$ with $R\le1$, is a minimizer of \eqref{main}. We enforce $u$ to be bounded in $B_{4R}$ provided $sp\le n$. Given $0\le \tau$, %\le\|u\|_{L^\infty(B_{4R})}
one can find a constant $C>1$, depending on $\textbf{data}(B_{4R})$ and $\lambda$, such that either
\begin{equation*}
%\label{4-4-1}
\mathrm{Tail} (u_-;x_0,4R)+\mathrm{Tail}_a(u_-;x_0,4R)\ge h_{4R}(\tau)
\end{equation*}
or
\begin{equation*}
%\label{4-4-2}
|\{u\le \tau\}\cap B_{2R}|\le C\left(\frac{\tau^{p-1}}{[\mathrm{Tail}_p(u_+;x_0,2R)]^{p-1}}+a^-_{4R}\frac{\tau^{q-1}}{[\mathrm{Tail}_q(u_+;x_0,2R)]^{q-1}}\right)|B_{2R}|,
\end{equation*}
where the definitions of $\mathrm{Tail}_p(u_+;x_0,2R)$ and $\mathrm{Tail}_q(u_+;x_0,2R)$ are identical to those in Theorem \ref{Thm3}.
\end{lemma}

\begin{proof}
Set $w_\pm:=(u-2\tau)_\pm$. We rewrite \eqref{4-3-3} as
\begin{equation}
\label{4-4-1}
\int_{B_{2R}}w_-(x)\left(\int_{\mathbb{R}^n}\frac{w_+^{p-1}(y)}{|x-y|^{n+sp}}+a(x,y)\frac{w_+^{q-1}(y)}{|x-y|^{n+tq}}\,dy\right)\,dx\le C H_{4R}(\tau)|B_{4R}|.
\end{equation}
Next, we shall differentiate two diverse scenarios in \eqref{4-3-4} as well.

\medskip

\textbf{Case $\eqref{4-3-4}_1$}. When $|x-x_0|\le2R$ and $|y-x_0|\ge2R$, we have $|x-y|\le2|y-x_0|$. Due to \eqref{4-4-1}, one obtains
\begin{align*}
C\frac{\tau^p}{(4R)^{sp}}|B_{4R}|&\ge C H_{4R}(\tau)|B_{4R}|\\
&\ge\int_{B_{2R}}w_-(x)\int_{\mathbb{R}^n}\frac{w_+^{p-1}(y)}{|x-y|^{n+sp}}\,dydx\\
&\ge\int_{B_{2R}}w_-(x)\int_{\mathbb{R}^n\setminus B_{2R}}\frac{\frac{1}{C}u_+^{p-1}(y)-\tau^{p-1}}{(2|y-x_0|)^{n+sp}}\,dydx\\
&\ge\frac{1}{C}\int_{B_{2R}}w_-(x)\,dx\cdot\int_{\mathbb{R}^n\setminus B_{2R}}\frac{u_+^{p-1}(y)}{|y-x_0|^{n+sp}}\,dy-C\frac{\tau^p}{(4R)^{sp}}|B_{2R}|,
\end{align*}
i.e.,
$$
\mathrm{Tail}(u_+;x_0,2R)\int_{B_{2R}}w_-(x)\,dx\le C\frac{\tau^p}{(4R)^{sp}}|B_{2R}|.
$$
Finally, it follows from the inequality \eqref{4-3-7} that
\begin{equation}
\label{4-4-2}
|B_{2R}\cap\{u\le\tau\}|\le\frac{C\tau^{p-1}}{(2R)^{sp}\mathrm{Tail}(u_+;x_0,2R)}|B_{2R}|.
\end{equation}

\textbf{Case $\eqref{4-3-4}_2$}. In an analogous manner to Case $\eqref{4-3-4}_1$ in this proof, we evaluate
\begin{align}
\label{4-4-3}
C a^-_{4R}\frac{\tau^q}{(4R)^{tq}}|B_{4R}|&\ge \int_{B_{2R}}w_-(x)\int_{\mathbb{R}^n}a(x,y)\frac{w_+^{q-1}(y)}{|x-y|^{n+tq}}\,dydx \nonumber\\
&\ge \int_{B_{2R}}w_-(x)\int_{\mathbb{R}^n\setminus B_{2R}}a(x,y)\frac{\frac{1}{C}u_+^{q-1}(y)-\tau^{q-1}}{(2|y-x_0|)^{n+tq}}\,dydx \nonumber\\
&=\frac{1}{C}\int_{B_{2R}}w_-(x)\int_{\mathbb{R}^n\setminus B_{2R}}a(x,y)\frac{u_+^{q-1}(y)}{|y-x_0|^{n+tq}}\,dydx \nonumber\\
&\quad-\frac{\tau^{q-1}}{C}\int_{B_{2R}}w_-(x)\int_{\mathbb{R}^n\setminus B_{2R}}\frac{a(x,y)}{|y-x_0|^{n+tq}}\,dydx \nonumber\\
&=:I_1-I_2.
\end{align}
Since $a(x,y)\ge\lambda$ in $\mathbb{R}^n\times\mathbb{R}^n$, we readily discover
\begin{equation}
\label{4-4-4}
I_1\ge \frac{\lambda}{C}\int_{B_{2R}}w_-(x)\,dx\int_{\mathbb{R}^n\setminus B_{2R}}\frac{u_+^{q-1}(y)}{|y-x_0|^{n+tq}}\,dy.
\end{equation}
\emph{Here, let us remark that if $a(x,y)\ge0$ instead of $a(x,y)\ge\lambda$, then we do not know how to find a proper nonlocal integral bounding $I_1$ from below so that we can divide the integrals with respect to $x$ and $y$.} Now we focus on dealing with $I_2$. For $x\in B_{2R}$ and $y\in\mathbb{R}^n\setminus B_{2R}$, by means of the H\"{o}lder continuity of $a(\cdot,\cdot)$, we can see
\begin{align*}
a(x,y)&=a(x,y)-a(x,x)+a(x,x)-a^-_{4R}+a^-_{4R}\\
&\le [a]_\alpha|x-y|^\alpha+a(x,x)-a^-_{4R}+a^-_{4R}\\
&\le 2[a]_\alpha|y-x_0|^\alpha+\sup_{B_{4R}\times B_{4R}}|a(x,y)-a(x',y')|+a^-_{4R}\\
&\le2[a]_\alpha|y-x_0|^\alpha+C[a]_\alpha R^\alpha+a^-_{4R}.
\end{align*}
Therefore, by virtue of $R\le1$ and $\lambda\le a^-_{4R}$,
\begin{align}
\label{4-4-5}
I_2&\le \frac{\tau^{q-1}}{C}\int_{B_{2R}}w_-(x)\,dx\int_{\mathbb{R}^n\setminus B_{2R}}\frac{2[a]_\alpha|y-x_0|^\alpha+C[a]_\alpha R^\alpha+a^-_{4R}}{|y-x_0|^{n+tq}}\,dy \nonumber\\
&\le\frac{\tau^{q}}{C}(R^{\alpha-tq}+a^-_{4R}R^{-tq})|B_{2R}| \nonumber\\
&\le \frac{a^-_{4R}+1}{C}\frac{\tau^q}{R^{tq}}|B_{2R}| \nonumber\\
&\le a^-_{4R}\frac{\lambda^{-1}+1}{C}\frac{\tau^q}{(4R)^{tq}}|B_{2R}|.
\end{align}
Combining \eqref{4-4-4} and \eqref{4-4-5} with \eqref{4-4-3} yields that
$$
\int_{B_{2R}}w_-(x)\,dx\int_{\mathbb{R}^n\setminus B_{2R}}\frac{u_+^{q-1}(y)}{|y-x_0|^{n+tq}}\,dy\le Ca^-_{4R}\frac{\tau^q}{(4R)^{tq}}|B_{2R}|.
$$
Via \eqref{4-3-7}, we arrive at
\begin{equation}
\label{4-4-6}
|\{u\le \tau\}\cap B_{2R}|\le Ca^-_{4R}\frac{\tau^{q-1}}{(2R)^{tq}\mathrm{Tail}_q(u_+;x_0,2R)}|B_{2R}|,
\end{equation}
where the constant $C>1$ depends on $\textbf{data}(B_{4R})$ and $\lambda$. As a consequence, we immediately deduce the desired measure estimate from \eqref{4-4-2} and \eqref{4-4-6}.
\end{proof}

\section{Proof of weak Harnack inequalities}
\label{sec5}

In this section, we complete in turn the proof of three weak Harnack estimates, Theorems \ref{Thm1}--\ref{Thm3}.

\medskip
\noindent{\textbf{Proof of Theorem \ref{Thm1}}}.
Let $B_{4R}(x_0)\subset\Omega$ with $R\le1$, and the minimizer $u$ be nonnegative and bounded in $B_{4R}(x_0)$. We will neglect the center $x_0$ in this proof. We assume $u\not\equiv0$ in $B_R$ or else there is nothing to prove. Let $\delta\in\left(0,\frac{1}{8}\right]$ be the parameter provided by Corollary \ref{cor4-2-2} with $\nu=\frac{1}{2}$. Set
$$
\beta:=\frac{1}{2\log_\frac{1}{2}\delta}\in(0,1).
$$
First, we claim that for $\tau\ge0$ there holds
\begin{equation}
\label{5-1-1}
\inf_{B_R(x_0)}u+h_{4R}^{-1}\big(\mathrm{Tail}(u_-;x_0,4R)+\mathrm{Tail}_a(u_-;x_0,4R)\big)
\ge\delta\left(\frac{|E^+(\tau,R)|}{|B_R|}\right)^\frac{1}{\beta}\tau,
\end{equation}
where the set $E^+(\tau,R):=B_R(x_0)\cap\{u>\tau\}$. Owing to the definition of $E^+(\tau,R)$, the case $\tau\ge\sup_{B_R}u$ is obvious since $u$ is locally bounded in $\Omega$. Then we only need to show this assertion for $\tau\in\left[0,\sup_{B_R}u\right)$. Now given a $\tau\in\left[0,\sup_{B_R}u\right)$, let $k$ be the smallest integer such that
\begin{equation}
\label{5-1-2}
|E^+(\tau,R)|\ge2^{-k}|B_R|,
\end{equation}
that is,
$$
\log _{\frac{1}{2}} \frac{|E^+(\tau,R)|}{|B_R|}
\leq k<1+\log _{\frac{1}{2}} \frac{|E^+(\tau,R)|}{|B_R|} .
$$
We further have
\begin{equation}
\label{5-1-3}
\delta^k\ge\delta\left(\frac{|E^+(\tau,R)|}{|B_R|}\right)^\frac{1}{\beta}.
\end{equation}
Due to $u\ge0$ in $B_{4R}$, \eqref{5-1-1} directly holds true provided
\begin{equation}
\label{5-1-4}
\mathrm{Tail}(u_-;x_0,4R)+\mathrm{Tail}_a(u_-;x_0,4R)\ge h_{4R}(\delta^k\tau).
\end{equation}
Otherwise, if \eqref{5-1-4} is false, then we combine \eqref{5-1-2} with the converse of \eqref{5-1-4} and employ Corollary \ref{cor4-2-2} to deduce
$$
u\ge\delta^k\tau  \quad\text{in }   B_R.
$$
All in all, we have
$$
\inf_{B_R}u+h_{4R}^{-1}\big(\mathrm{Tail}(u_-;x_0,4R)+\mathrm{Tail}_a(u_-;x_0,4R)\big)\ge\delta^k\tau,
$$
and from \eqref{5-1-3} get the inequality \eqref{5-1-1}.

Next, we denote
$$
L:=\inf_{B_R}u+h_{4R}^{-1}\big(\mathrm{Tail}(u_-;x_0,4R)+\mathrm{Tail}_a(u_-;x_0,4R)\big)
$$
and rearrange \eqref{5-1-1} as
$$
\left(\frac{L}{\delta\tau}\right)^\beta\ge\frac{|E^+(\tau,R)|}{|B_R|}.
$$
Using the last display and Cavalieri's principle, we calculate
\begin{align*}
\mint_{B_R}u^\frac{\beta}{2}(x)\,dx&=\frac{\beta}{2}\int^\infty_0\tau^{\frac{\beta}{2}-1}\frac{|E^+(\tau,R)|}{|B_R|}\,d\tau\\
&\le\frac{\beta}{2}\left[\int^L_0\tau^{\frac{\beta}{2}-1}\,d\tau+\left(\frac{L}{\delta}\right)^\beta\int^\infty_L\tau^{\frac{\beta}{2}-1-\beta}\,d\tau\right]\\
&\le (1+\delta^{-\beta})L^\frac{\beta}{2}.
\end{align*}
Therefore, we have finished the proof of Theorem \ref{Thm1}. \hfill $\Box$

\bigskip

\noindent{\textbf{Proof of Theorem \ref{Thm2}}}.
Let $B_{4R}(x_0)\subset\Omega$ with $R\le1$, and the minimizer $u$ be nonnegative and bounded in $B_{4R}(x_0)$. We also omit the center $x_0$ later. Let the parameter $\nu\in(0,1)$ be from Corollary \ref{cor4-2-1}. Moreover, according to Corollary \ref{cor4-2-1} with the constant $C\ge1$ determined there, we select $\tau>0$ to satisfy
$$
C\left(\frac{\tau^{p-1}}{(u^{p-1})_{B_{2R}}}+\frac{\tau^{q-1}}{(u^{q-1})_{B_{2R}}}\right)\le\nu.
$$
Then we further choose such $\tau>0$ that
$$
\frac{\tau^{p-1}}{(u^{p-1})_{B_{2R}}}\le\frac{\nu}{2C} \quad\text{and}\quad \frac{\tau^{q-1}}{(u^{q-1})_{B_{2R}}}\le\frac{\nu}{2C},
$$
namely,
\begin{equation}
\label{5-2-1}
\tau\le\left(\frac{\nu}{2C}(u^{p-1})_{B_{2R}}\right)^\frac{1}{p-1}   \quad\text{and}\quad \tau\le\left(\frac{\nu}{2C}(u^{q-1})_{B_{2R}}\right)^\frac{1}{q-1}.
\end{equation}
Thus we take
\begin{align*}
\tau=\left(\frac{\nu}{2C}\right)^\frac{1}{p-1}\min\left\{(u^{p-1})_{B_{2R}}^\frac{1}{p-1},(u^{q-1})_{B_{2R}}^\frac{1}{q-1}\right\}%\\
%&\le\min\left\{\left(\frac{\nu}{2C}(u^{p-1})_{B_{2R}}\right)^\frac{1}{p-1},\left(\frac{\nu}{2C}(u^{q-1})_{B_{2R}}\right)^\frac{1}{q-1}\right\}.
\end{align*}
to fulfill \eqref{5-2-1}. At this moment, through applying Corollary \ref{cor4-2-1}, we draw a conclusion that either
$$
\mathrm{Tail}(u_-;x_0,4R)+\mathrm{Tail}_a(u_-;x_0,4R)\ge h_{4R}(\tau)
$$
or
$$
u\ge\frac{1}{2}\tau \quad\text{in } B_R.
$$
As a consequence, it holds that
$$
\inf_{B_R}u+h_{4R}^{-1}\big(\mathrm{Tail}(u_-;x_0,4R)+\mathrm{Tail}_a(u_-;x_0,4R)\big)\ge\frac{3}{2}\tau.
$$
Finally, the choice of $\tau$ and this inequality directly imply the desired result. \hfill $\Box$

\bigskip

\noindent{\textbf{Proof of Theorem \ref{Thm3}}}.
Let $B_{4R}(x_0)\subset\Omega$ with $R\le1$, and the minimizer $u$ be nonnegative and bounded in $B_{4R}(x_0)$. Let the number $\nu\in(0,1)$ come from Corollary \ref{cor4-2-1}. Now pick a suitable $\tau>0$ such that
\begin{equation}
\label{5-3-1}
C\left(\frac{\tau^{p-1}}{[\mathrm{Tail}_p(u_+;x_0,2R)]^{p-1}}
+a^-_{4R}\frac{\tau^{q-1}}{[\mathrm{Tail}_q(u_+;x_0,2R)]^{q-1}}\right)\le\nu    %\frac{a^-_{4R}}{\lambda}
\end{equation}
with the constant $C\ge1$ fixed in Lemma \ref{lem4-4}.

We next select $\tau>0$ so small that
$$
\frac{\tau^{p-1}}{[\mathrm{Tail}_p(u_+;x_0,2R)]^{p-1}}\le\frac{\nu}{2C} \quad\text{and}\quad a^-_{4R}\frac{\tau^{q-1}}{[\mathrm{Tail}_q(u_+;x_0,2R)]^{q-1}}\le\frac{\nu}{2C}, %\frac{a^-_{4R}}{\lambda}
$$
i.e.,
$$
\tau\le\left(\frac{\nu}{2C}\right)^\frac{1}{p-1}\mathrm{Tail}_p(u_+;x_0,2R)   \quad\text{and}\quad \tau\le\left(\frac{\nu}{2Ca^-_{4R}}\right)^\frac{1}{q-1}\mathrm{Tail}_q(u_+;x_0,2R).
$$
According to the condition \eqref{5-3-1} and Corollary \ref{cor4-2-1}, we deduce that either
$$
\mathrm{Tail}(u_-;x_0,4R)+\mathrm{Tail}_a(u_-;x_0,4R)\ge h_{4R}(\tau)
$$
or
$$
u\ge\frac{1}{2}\tau \quad\text{in } B_R(x_0).
$$
By taking
\begin{align*}
\tau=\left(\frac{\nu}{2C}\right)^\frac{1}{p-1}\min\left\{\mathrm{Tail}_p(u_+;x_0,2R),
\left(\frac{1}{a^-_{4R}}\right)^\frac{1}{q-1}\mathrm{Tail}_q(u_+;x_0,2R)\right\},
\end{align*}
the eventual estimate in Theorem \ref{Thm3} follows readily. \hfill $\Box$

\section{Boundedness on minimizers}
\label{sec6}

In the terminal section, we are ready to demonstrate local boundedness (Theorem \ref{Thm4}) on the minimizers of \eqref{main}, and subsequently present further the local sup-estimate. To this aim, we first give some notations. Let $B_r(x_0)\subset\Omega$ be any ball with $0<r\le1$. Throughout this section, set separately sequences of radii, of balls, of levels and of truncation functions
$$
r_i=\frac{r}{2}+\frac{r}{2^{i+1}}, \quad  B_i:=B_{r_i}(x_0), \quad   k_i=2(1-2^{-i-1})\overline{k}  \quad \text{and} \quad  w_i:=(u-k_i)_+
$$
with $\overline{k}>0$, and note
$$
r_i-r_{i+1}=2^{-i-2}r, \quad k_{i+1}-k_i=2^{-i}\overline{k} \quad \text{and} \quad   w_{i+1}\le w_i\le u_+
$$
for $i=0,1,2,\cdots$.

The proof of local boundedness on the minimizer $u$ of \eqref{main} is indeed similar to that of \cite[Theorem 1.1]{BOS22}, but we still provide the inferring processes because of weakening the condition $a\in L^\infty(\mathbb{R}^n\times\mathbb{R}^n)$ in \cite[Theorem 1.1]{BOS22} as $a\in L^\infty_{\rm loc}(\Omega\times\Omega)$ here.

\medskip

\noindent{\textbf{Proof of Theorem \ref{Thm4}}}.
Define
$$
a^+_r:=\sup_{B_r(x_0)\times B_r(x_0)}a(x,y) \quad\text{and}\quad  H^+_r(\tau):=\tau^p+a^+_r\tau^q.
$$
We now apply Lemma \ref{lem3-1} with $\rho:=r_{i+1}$, $r:=r_i$ and $w_+:=w_{i+1}$ to get
\begin{align}
\label{6-1}
&\quad\mint_{B_{i+1}}\int_{B_{i+1}}\frac{|w_{i+1}(x)-w_{i+1}(y)|^p}{|x-y|^{n+sp}}\,dxdy  \nonumber\\
&\leq C\Bigg[\frac{2^{iq}}{r^{tq}}\mint_{B_i}(w_{i+1}^p+a^+_rw_{i+1}^q)\,dx
+2^{(n+tq)i}\mint_{B_i}w_{i+1}\,dx  \nonumber\\
&\qquad\cdot\left(\int_{\mathbb{R}^n\setminus B_{\frac{r}{2}}}\frac{u^{p-1}_+(x)}{|x-x_0|^{n+sp}}\,dx+\sup_{y\in B_r}
\int_{\mathbb{R}^n\setminus B_{\frac{r}{2}}}a(x,y)\frac{u^{q-1}_+(x)}{|x-x_0|^{n+tq}}\,dx\right)\Bigg]    \nonumber\\
&\leq C\Bigg[\frac{2^{iq}}{r^{tq}}\mint_{B_i}(w_{i+1}^p+a^+_rw_{i+1}^q)\,dx+2^{(n+tq+p)i}\frac{T\Big(u_+;x_0,\frac{r}{2}\Big)}{\overline{k}^{p-1}}
\mint_{B_i}w_{i}^p\,dx\Bigg]   \nonumber\\
&\le C2^{(n+2q)i}\left(\frac{1}{r^{tq}}+\frac{T\Big(u_+;x_0,\frac{r}{2}\Big)}{\overline{k}^{p-1}}\right) \mint_{B_i} H^+_r(w_i)\,dx,
\end{align}
where the notation $T\Big(u_+;x_0,\frac{r}{2}\Big)$ is given as \eqref{6-1-2}. Here in the penultimate inequality, we have used
$$
\int_{B_i}w_{i+1}\,dx\le\int_{B_i}(u-k_i)_+\left(\frac{(u-k_i)_+}{k_{i+1}-k_i}\right)^{p-1}\,dx\le\frac{2^{ip}}{\overline{k}^{p-1}}\int_{B_i}w_i^p\,dx.
$$
It follows immediately from Lemma \ref{lem2-2} that
\begin{align}
\label{6-2}
&\quad r_{i+1}^{-sp}\mint_{B_{i+1}}H^+_r(w_{i+1})\,dx \nonumber\\
&\le\mint_{B_{i+1}}\left(\frac{w_{i+1}}{r_{i+1}^s}\right)^p+a^+_r\left(\frac{w_{i+1}}{r_{i+1}^t}\right)^q\,dx  \nonumber\\
&\leq Cr_{i+1}^{(s-t)q}a^+_r\left(\mint_{B_{i+1}}\int_{B_{i+1}}\frac{|w_{i+1}(x)-w_{i+1}(y)|^p}{|x-y|^{n+sp}}\,dxdy\right)^\frac{q}{p}  \nonumber\\
&\quad+C\left(\frac{|{\rm supp}\,w_{i+1}|}{|B_{i+1}|}\right)^\frac{sp}{n}\mint_{B_{i+1}}\int_{B_{i+1}}\frac{|w_{i+1}(x)-w_{i+1}(y)|^p}{|x-y|^{n+sp}}\,dxdy    \nonumber\\
&\quad+C\left(\frac{|{\rm supp}\,w_{i+1}|}{|B_{i+1}|}\right)^{p-1}\mint_{B_{i+1}}\left(\frac{w_{i+1}}{r_{i+1}^s}\right)^p+a^+_r\left(\frac{w_{i+1}}{r_{i+1}^t}\right)^q\,dx.
\end{align}
In view of \eqref{6-1} and
$$
|{\rm supp}\,w_{i+1}|\le \int_{B_{i+1}\cap\{u\ge k_{i+1}\}}\frac{(u-k_i)_+^p}{(k_{i+1}-k_i)^p}\,dx\le \frac{2^{ip}}{\overline{k}^{p}}\int_{B_i}w_i^p\,dx,
$$
the display \eqref{6-2} turns into
\begin{align}
\label{6-3}
\mint_{B_{i+1}}H^+_r(w_{i+1})\,dx&\le C2^\frac{q(n+2q)i}{p}r^{sp+(s-t)q}\left(\frac{1}{r^{tq}}+\frac{T\Big(u_+;x_0,\frac{r}{2}\Big)}{\overline{k}^{p-1}}\right)^\frac{q}{p}
\left(\mint_{B_i}H^+_r(w_i)\,dx\right)^\frac{q}{p}  \nonumber\\
&\quad+C2^{\left(n+tq+\frac{sp^2}{n}\right)i}\frac{1+r^{tq}T\Big(u_+;x_0,\frac{r}{2}\Big)/{\overline{k}^{p-1}}}{r^{tq-sp}\overline{k}^\frac{sp^2}{n}}
\left(\mint_{B_i}H^+_r(w_i)\,dx\right)^{1+\frac{sp}{n}} \nonumber\\
&\quad+C2^{ip^2}\frac{r^{sp-tq}}{\overline{k}^{p(p-1)}}\left(\mint_{B_i}H^+_r(w_i)\,dx\right)^p.
\end{align}

Next, let us define
$$
Y_i=\mint_{B_i}H^+_r(w_i)\,dx.
$$
Due to $H_r(u)\in L^1(\Omega)$ from \eqref{thm4}, we see that
$$
Y_0=\mint_{B_r}H^+_r((u-k_0)_+)\,dx\rightarrow0  \quad \text{as } k_0\rightarrow\infty.
$$
Thereby, we can first take large $\overline{k}\ge[T(u_+;x_0,r/2)]^\frac{1}{p-1}+1$ to meet
$$
Y_i\le Y_{i-1}\le \cdots\le Y_1\le Y_0\le1
$$
for $i=1,2,3,\cdots$. At this point, \eqref{6-3} becomes
\begin{align*}
Y_{i+1}&\le C\left(2^\frac{q(n+2q)i}{p}Y_i^\frac{q}{p}+2^{\left(n+tq+\frac{sp^2}{n}\right)i}Y_i^{\left(1+\frac{sp}{n}\right)}+2^{ip^2}Y_i^p\right)\\
&\le C2^{i\theta}Y_i^{1+\sigma},
\end{align*}
where
$$
\theta=\max\left\{\frac{q(n+2q)}{p},n+tq+\frac{sp^2}{n},p^2\right\}  \quad\text{and} \quad \sigma=\min\left\{\frac{q}{p}-1,\frac{sp}{n},p-1\right\},
$$
and $C>1$ depends on $n, p, q, s, t, \Lambda, \|a\|_{L^\infty(B_r)}, r$ and $T\Big(u_+; x_0, \frac{r}{2}\Big)$. In the end, we choose again such large $\overline{k}\ge[T(u_+; x_0, r/2)]^\frac{1}{p-1}+1$ that
$$
Y_0\le C^{-\frac{1}{\sigma}}2^{-\frac{\theta}{\sigma^2}},
$$
and further invoke a geometric convergence lemma, Lemma \ref{lem-2-3}, to derive $Y_i\rightarrow0$ as $i\rightarrow\infty$, which indicates the local boundedness from above. In a similar way to the previous processes for $-u$, we will get the boundedness from below and then have $u\in L^\infty(B_{r/2})$.  \hfill $\Box$

\medskip

At the end of this manuscript, we are going to justify the supremum estimate, of independent interest, on the minimizer of \eqref{main}, Theorem \ref{Thm5}, which possesses the potential application for exploring Harnack inequality. %Before this, we introduce some notations as follows. Let
%$$
%\mathfrak{h}_r(\tau)=\frac{\tau^{p-1}}{r^{sp}}+a^-_r\frac{\tau^{q-1}}{r^{tq}} \quad\text{and}\quad  \mathcal{H}_r(\tau)=\frac{\tau^p}{r^{sp}}+a^-_r\frac{\tau^q}{r^{tq}}.
%$$
%Furthermore, $\mathfrak{h}^{-1}_r$ and $\mathcal{H}^{-1}_r$ means the inverse of $\mathfrak{h}_r$ and $\mathcal{H}_r$ respectively.

\medskip

\noindent\textbf{Proof of Theorem \ref{Thm5}}.
Let $B_r:=B_r(x_0)$ later. Recalling $\frac{r}{2}\le r_{i+1}<r_i\le r$ and $w_{i+1}\le w_i\le u_+$, and exploiting Lemma \ref{lem3-2} with $\rho:=r_{i+1}$, $r:=r_i$ and $w_+:=w_{i+1}$, we arrive at
\begin{align}
\label{t5-1}
&\quad\mint_{B_{i+1}}\int_{B_{i+1}}\frac{|w_{i+1}(x)-w_{i+1}(y)|^p}{|x-y|^{n+sp}}+a_r^-\frac{|w_{i+1}(x)-w_{i+1}(y)|^q}{|x-y|^{n+tq}}\,dxdy  \nonumber\\
&\leq C2^{i(n+q)}\Bigg[\mint_{B_{i}}\left(\frac{w_{i+1}^p}{r^{sp}_i}+a_r^-\frac{w_{i+1}^q}{r^{tq}_i}\right)dx+\mint_{B_i}w_{i+1}\,dx  \nonumber\\
&\qquad\qquad\quad\cdot\left(\int_{\mathbb{R}^n\setminus B_{\frac{r}{2}}}\frac{u^{p-1}_+(x)}{|x-x_0|^{n+sp}}\,dx+\sup_{y\in B_r}
\int_{\mathbb{R}^n\setminus B_{\frac{r}{2}}}a(x,y)\frac{u^{q-1}_+(x)}{|x-x_0|^{n+tq}}\,dx\right)\Bigg]    \nonumber\\
&\leq C2^{i(n+q)}\left[\mint_{B_{i}}H_r(w_i)\,dx+\frac{T\Big(u_+;x_0,\frac{r}{2}\Big)}{h_r(k_{i+1}-k_i)}\mint_{B_{i}}H_r(w_i)\,dx\right],
\end{align}
where we have used
$$
\int_{B_i}w_{i+1}\,dx\le \int_{B_i}w_{i+1}\frac{h_r((u-k_i)_+)}{h_r(k_{i+1}-k_i)}\,dx\le\int_{B_i}\frac{H_r(w_i)}{h_r(k_{i+1}-k_i)}\,dx.
$$

Next, let the number $\beta>1$ be defined as in \eqref{4-2-3} and $\beta'=\frac{\beta}{\beta-1}$. Let us first consider the case $a^-_r>0$. At this moment, we know that $u$ belongs to $W^{s,p}(B_r)\cap W^{t,q}(B_r)$. We apply the fractional Sobolev embedding theorem to discover
\begin{align}
\label{t5-2}
&\quad\mint_{B_{i+1}}H_r(w_{i+1})\,dx \nonumber\\
&\leq \left(\mint_{B_{i+1}}H_r^\beta(w_{i+1})\,dx\right)^\frac{1}{\beta}\left(\frac{|E^+(k_{i+1},r_{i+1})|}{|B_{i+1}|}\right)^\frac{1}{\beta'}  \nonumber\\
&\leq \left[\frac{1}{r^{sp}}\left(\mint_{B_{i+1}}w_{i+1}^{p\beta}\,dx\right)^\frac{1}{\beta}+\frac{a^-_r}{r^{tq}}\left(\mint_{B_{i+1}}w_{i+1}^{q\beta}\,dx\right)^\frac{1}{\beta}\right]
\left(\frac{|E^+(k_{i+1},r_{i+1})|}{|B_{i+1}|}\right)^\frac{1}{\beta'} \nonumber\\
&\le C\Bigg[\mint_{B_{i+1}}\int_{B_{i+1}}\frac{|w_{i+1}(x)-w_{i+1}(y)|^p}{|x-y|^{n+sp}}+a_r^-\frac{|w_{i+1}(x)-w_{i+1}(y)|^q}{|x-y|^{n+tq}}\,dxdy \nonumber\\
&\quad+\mint_{B_{i+1}}\left(\frac{w_{i+1}^p}{r^{sp}}+a_r^-\frac{w_{i+1}^q}{r^{tq}}\right)dx\Bigg]\left(\frac{|E^+(k_{i+1},r_{i+1})|}{|B_{i+1}|}\right)^\frac{1}{\beta'}
\end{align}
with $E^+(k_{i+1},r_{i+1})=\{u>k_{i+1}\}\cap B_{i+1}$. Observe that
\begin{equation}
\label{t5-3}
|E^+(k_{i+1},r_{i+1})|\le \int_{B_i}\frac{H_r(w_i)}{H_r(k_{i+1}-k_i)}\,dx.
\end{equation}
Thanks to the displays \eqref{t5-1}--\eqref{t5-3},
\begin{align*}
\mint_{B_{i+1}}H_r(w_{i+1})\,dx&\leq C\frac{2^{i(n+q)}}{H_r^\frac{1}{\beta'}(k_{i+1}-k_i)}\left(1+\frac{T\Big(u_+;x_0,\frac{r}{2}\Big)}{h_r(k_{i+1}-k_i)}\right)
\left(\mint_{B_{i}}H_r(w_{i})\,dx\right)^{1+\frac{1}{\beta'}}  \nonumber\\
&\leq C\frac{2^{i(n+3q)}}{H_r^\frac{1}{\beta'}(\overline{k})}\left(1+\frac{T\Big(u_+;x_0,\frac{r}{2}\Big)}{h_r(\overline{k})}\right)
\left(\mint_{B_{i}}H_r(w_{i})\,dx\right)^{1+\frac{1}{\beta'}}.
\end{align*}
Denote
$$
Y_i=\mint_{B_{i}}H_r(w_{i})\,dx.
$$
We now choose $\overline{k}>0$ so large that
\begin{equation}
\label{t5-4}
\frac{T\Big(u_+;x_0,\frac{r}{2}\Big)}{h_r(\overline{k}/\delta)}\le\frac{\delta^pT\Big(u_+;x_0,\frac{r}{2}\Big)}{h_r(\overline{k})}\le1.
\end{equation}
Then we will get
$$
Y_{i+1}\le\frac{C2^{i(n+3q)}}{\delta^pH_r^\frac{1}{\beta'}(\overline{k})}Y_i^{1+\frac{1}{\beta'}}.
$$
Through employing Lemma \ref{lem-2-3}, we know that if
\begin{equation}
\label{t5-5}
Y_0\le\left(\frac{C}{\delta^pH_r^\frac{1}{\beta'}(\overline{k})}\right)^{-\beta'}2^{-(n+3q)\beta'^2},
\end{equation}
then it holds that $Y_i\rightarrow0$ as $i\rightarrow \infty$. Meanwhile, we get $u\le 2\overline{k}$ in $B_\frac{r}{2}$. The remaining work is to select further $\overline{k}>0$ large enough such that \eqref{t5-5} is true. Notice
$$
Y_0=\mint_{B_{r}}H_r((u-\overline{k})_+)\,dx\le\mint_{B_{r}}H_r(u_+)\,dx.
$$
Hence we just need to require
$$
\mint_{B_r}H_r(u_+)\,dx\le \left(\frac{C}{\delta^pH_r^\frac{1}{\beta'}(\overline{k})}\right)^{-\beta'}2^{-(n+3q)\beta'^2},
$$
namely,
$$
\overline{k}\ge H^{-1}_r\left[\left(\frac{C}{\delta^p}\right)^{\beta'}2^{(n+3q)\beta'^2}\mint_{B_r}H_r(u_+)\,dx\right].
$$
Finally, combining the last inequality and \eqref{t5-4}, we could take
$$
\overline{k}=H^{-1}_r\left[\left(\frac{C}{\delta^p}\right)^{\beta'}2^{(n+3q)\beta'^2}\mint_{B_r}H_r(u_+)\,dx\right]+\delta
h_r^{-1}\Big[T\Big(u_+;x_0,\frac{r}{2}\Big)\Big].
$$
The proof is complete now.   \hfill $\Box$

\section*{Acknowledgements}
This work was supported by the National Natural Science Foundation of China (Nos. 12071098 and 12301245), the National Postdoctoral Program for Innovative Talents of China (No. BX20220381), the Young talents sponsorship program of Heilongjiang Province (No. 2023QNTJ004) and the Fundamental Research Funds for the Central Universities (No. 2022FRFK060022).

\section*{Declarations}
\subsection*{Conflict of interest} The authors declare that there is no conflict of interest. We also declare that this
manuscript has no associated data.

\subsection*{Data Availability} Data sharing is not applicable to this article as no datasets were generated or analysed
during the current study.

\end{document}